\documentclass[a4 paper]{article}

\usepackage{fancyhdr}
\usepackage{latexsym}
\usepackage{mathrsfs}
\usepackage{cancel}
\usepackage{color}
\usepackage{multirow}
\usepackage{eso-pic}
\usepackage{dsfont}
\usepackage{amssymb}
\usepackage{graphicx}
\usepackage{setspace}
\usepackage{geometry}
\usepackage{amsthm}
\usepackage{hyperref}
\usepackage[intlimits]{amsmath}
\usepackage{hyperref}
\usepackage{tikz-cd}
\usepackage{amssymb,amsfonts}
\usepackage[all,arc]{xy}
\usepackage{enumerate}
\usepackage{mathrsfs}
\usepackage{tikz}
\usepackage{bbm}
\usepackage{graphicx}
\usepackage{mathabx,epsfig}
\def\Lacts{\mathrel{\reflectbox{$\righttoleftarrow$}}}
\def\Racts{\mathrel{\reflectbox{$\lefttorightarrow$}}}
\usepackage{authblk}

\newcommand{\wt}{\widetilde}
\newcommand{\wh}{\widehat}

\newcommand{\Hom}{\mathrm{Hom}}

\newcommand{\Mod}{\mathrm{-Mod}}

\newcommand{\lp}{\left(}
\newcommand{\rp}{\right)}

\newcommand{\lpp}{(\!(}
\newcommand{\rpp}{)\!)}
\newcommand{\lbb}{[\![}
\newcommand{\rbb}{]\!]}

\newcommand{\pd}{{\partial}}

\newcommand{\Ch}{\mathbb{C}\lbb \hbar\rbb}

\newcommand{\Chtensor}{\otimes_{\Ch}}

\newcommand{\C}{\mathbb C}
\newcommand{\R}{\mathbb R}
\newcommand{\Z}{\mathbb Z}

\newcommand{\E}{\mathbb E}

\newcommand{\fg}{\mathfrak{g}}
\newcommand{\fh}{\mathfrak{h}}

\newcommand{\fd}{\mathfrak{d}}

\newcommand{\CA}{{\mathcal A}}

\newcommand{\CD}{{\mathcal D}}

\newcommand{\CK}{{\mathcal K}}
\newcommand{\CL}{{\mathcal L}}

\newcommand{\CN}{{\mathcal N}}
\newcommand{\CO}{{\mathcal O}}

\newcommand{\CV}{{\mathcal V}}

\newcommand{\be}{\begin{equation}}
\newcommand{\ee}{\end{equation}}

\newcommand{\btik}{\begin{tikzcd}}
\newcommand{\etik}{\end{tikzcd}}

\begin{document}

\title{1-shifted Lie bialgebras and their quantizations}

\author[1]{Wenjun Niu}
\author[2]{Victor Py}

\affil[1]{Perimeter Institute for Theoretical Physics, 31 Caroline St N, Waterloo, ON N2L 2Y5, Canada.}

\affil[2]{School of Mathematics, University of Edinburgh, James Clerk Maxwell Building, Peter Guthrie Tait Rd., Edinburgh EH9 3FD, United Kingdom.}

\newtheorem{Def}{Definition}[section]
\newtheorem{Thm}[Def]{Theorem}
\newtheorem{Prop}[Def]{Proposition}
\newtheorem{Cor}[Def]{Corollary}
\newtheorem{Lem}[Def]{Lemma}

\theoremstyle{definition}
\newtheorem{Exp}[Def]{Example}
\newtheorem{Rem}[Def]{Remark}
\newtheorem{Asp}[Def]{Assumption}

\numberwithin{equation}{section}
\renewcommand{\Affilfont}{\mdseries \fontsize{9}{10.8} \itshape}
\maketitle

\abstract{In this paper, we define (cohomologically) 1-shifted Manin triples and 1-shifted Lie bialgebras, and study their properties. We derive many results that are parallel to those found in ordinary Lie bialgebras, including the double construction and the existence of a 1-shifted $r$-matrix satisfying the classical Yang-Baxter equation.  

Turning to quantization, we first construct a canonical quantization for each 1-shifted metric Lie algebra $\fg$, producing a deformation to the symmetric monoidal category of $\fg$ modules over a formal variable $\hbar$. This quantization is in terms of a curved differential graded algebra. Under a further technical assumption, we construct quantizations of transverse Lagrangian subalgebras of $\fg$, which form a pair of DG algebras connected by Koszul duality, and give rise to monoidal module categories of the quantized double.  

Finally, we apply this to Manin triples arising from Lie algebras of loop groups, and construct 1-shifted meromorphic $r$-matrices. The resulting quantizations are the cohomologically-shifted analogue of Yangians.}

\newpage

\tableofcontents

\newpage

\section{Introduction}

Given a Lie bialgebra $\fh$, it is well-known that one can construct its classical double $\fg=D(\fh)$ and the associated Manin-triple $(\fg, \fh, \fh^*)$. In \cite{etingof1996quantization}, Etingof and Kazhdan showed that any such Lie bialgebra admits a canonical quantization. The main ingredients in this construction are Drinfeld's quantization of $\fg$ as a quasi-triangular quasi-Hopf algebra \cite{drinfeld1991quasi}; and the strategy of \cite{kazhdan1993tensor}, which results in a fiber functor for Drinfeld's category. Moreover, the fiber functor constructed this way provides a polarization of the $R$-matrix, which leads to the Hopf subalgebra $U_\hbar (\fh)$. They further applied their results in \cite{etingof1998quantizationliebialgebrasiii, etingof1998quantizationliebialgebrasiv} to Lie bialgebras arising from loop groups, and constructed examples of quantum vertex algebras. 

On the other hand, in \cite{sBV}, the author defined and studied some properties of Lie bialgebras in homologically shifted settings. It is interesting to ask then whether these shifted Lie bialgebras admit quantizations, and what kind of algebraic structures they quantize to.  In this work, we focus on the case of \textit{1-shifted Lie bialgebras}. We will study the classical structure related to 1-shifted Lie bialgebras, and propose a construction of the quantization.

We will start with the definition and basic properties of 1-shifted Lie bialgebras. Although this is partially done in \cite{sBV}, we will start from scratch, partly because the analysis of \cite{sBV} is incomplete, and partly because we would like this paper to be self-contained. Instead of directly defining a 1-shifted Lie bialgebra, our starting point is its double. Namely, we consider a \textit{1-shifted metric Lie algebra}, which is a graded\footnote{Every vector space in this paper is $\Z$ graded and all structures respect this grading.} Lie algebra $\fg$ with a non-degenerate invariant bilinear form $\kappa$ of degree $1$, which we assume to be skew-symmetric for degree reasons, \textit{cf}. Section \ref{subsec:conventions}.  A Lagrangian Lie subalgebra of $\fg$ is a Lie subalgebra $\fh$ that is both isotropic and co-isotropic with respect to $\kappa$. We define a \textit{1-shifted Manin-triple} to be a 1-shifted metric Lie algebra $\fg$ together with two Lanrangian Lie subalgebras $\fh_\pm$ such that $\fg=\fh_+\oplus \fh_-$ (we call them transverse Lagrangian Lie subalgebras).  We study the induced co-bracket on $\fh_\pm$ from this set-up, which is ultimately the desired 1-shifted co-bracket. Our analysis, in particular Proposition \ref{Prop:delta}, leads to the following definition.

\begin{Def}[Definition \ref{Def:LBAmain}]\label{Def:LBAintro}

 A 1-shifted Lie bi-algebra is a Lie algebra $\fh$, equipped with a co-bracket of the form
 \be
 \delta: \fh\to \mathrm{Sym}^2(\fh)[1]
 \ee
satisfying
\be\label{eq:shiftedbialg}
\delta \otimes 1(\delta)+1\otimes \delta (\delta)=0\in \mathrm{Sym}^3(\fh), \qquad \delta([X,Y])=[\delta(X), \Delta(Y)]+(-1)^{|X|}[\Delta(X), \delta(Y)],
\ee
for all $X, Y\in\fh$. 
    
\end{Def}

We then  show that many results regarding Lie bialgebras and their doubles remain true in the shifted setting.

\begin{Thm}[Theorem \ref{Thm:Manindelta}, Proposition \ref{Prop:coboundaryr} \& Proposition \ref{Prop:1shiftedYB}]\label{Thm:Classintro}
    There is a one-to-one correspondence between 1-shifted Lie bialgebras and 1-shifted Manin-triples. For each 1-shifted Lie bialgebra $\fh$ and the associated Manin-triple $(\fg, \fh, \fh^*[-1])$, there exists a tensor (which we call the 1-shifted $r$-matrix)
    \be
\mathbf{r}\in \fh\otimes \fh^*[-1]
    \ee
    whose commutator induces the co-brackets on $\fh$ and $\fh^*[-1]$. Moreover, $\mathbf{r}$ satisfies classical Yang-Baxter equation. 
    
\end{Thm}

 Our next goal is to propose a quantization of 1-shifted Lie bialgebras, which is guided by some physics intuitions to be discussed in Sections \ref{subsec:TQFT} and \ref{subsec:HTQFT}. For each 1-shifted Manin-triple $(\fg, \fh_\pm)$, we will first define a canonical quantization of the double $\fg$. This quantization will be in terms of a \textit{curved differential graded algebra} (CDGA for short), and produces a monoidal category deforming the symmetric monoidal category $U(\fg)\Mod$. More specifically, we prove the following result.

 \begin{Thm}[Section \ref{subsec:quantdouble}]
     Let $\fg$ be a finite-dimensional 1-shifted metric Lie algebra. The following statements are true. 

     \begin{itemize}
     
         \item There exists a central element $W\in U(\fg)$ and an anti-symmetric invariant tensor
         \be
\Omega\in \fg\otimes \fg, \qquad \sigma\Omega=-\Omega,
         \ee
        both of which are canonically associated to $\fg$ and the bilinear form. 

         \item The CDGA $(U(\fg)\lbb\hbar\rbb, \hbar^2 W)$ is a co-algebra object in the category of CDGAs, whose coproduct is given by $(\Delta, \hbar\Omega)$, where $\Delta$ is the symmetric coproduct of $\fg$. In particular, $(U(\fg)\lbb\hbar\rbb, \hbar^2 W)\Mod$ is a monoidal category. 
         
     \end{itemize}
 \end{Thm}

 Turning to quantization of Lagrangians, we will impose the condition that $\fg_2=0$ in Assumption \ref{Asp:g2=0}. With this condition, we show that $U(\fh_\pm)$ naturally deform to DG algebras over $\C\lbb\hbar\rbb$, and their categories admit monoidal actions of $(U(\fg)\lbb\hbar\rbb, \hbar^2 W)\Mod$.  

\begin{Thm}[Section \ref{subsec:quantization}]\label{Thm:quantizeintro}
   Assume that $\fg_2=0$. The following are true. 

    \begin{itemize}
        \item There are DG algebras $U_\hbar (\fh_\pm)$ over the formal ring $\Ch$, whose limit at $\hbar=0$ give $U(\fh_\pm)$. 

        \item Let $U_\hbar (\hbar\fh_\pm)$ the DG subalgebra of $U_\hbar (\fh_\pm)$ generated by $\hbar\fh_\pm$. The limit of $U_\hbar (\hbar\fh_\pm)$ at $\hbar=0$ give the Chevalley-Eilenberg cochain complex of $\fh_\mp$. 

        \item The DG algebras $U_\hbar (\fh_+)$ and $U_\hbar (\hbar\fh_-)$ are Koszul dual of each other. 

        \item $U_\hbar (\fh_\pm)\Mod$ are monoidal module categories of $(U(\fg)\lbb\hbar\rbb, \hbar^2 W)\Mod$. 

    \end{itemize}
    
\end{Thm}

We note here that unlike the case of ordinary Lie bialgebras, the 1-shifted $r$-matrix $\mathbf{r}$ serves as a Maurer-Cartan element in the tensor product $U (\fg)\otimes U (\fg)$, and its polarization leads to the monoidal module structure of $U_\hbar (\fh_\pm)\Mod$. The Maurer-Cartan equation of $\mathbf{r}$ is a consequence of the Yang-Baxter equation. The natural interpretation of this quantization in terms of physical TQFTs is given in Section \ref{subsec:TQFT}. We also give a short discussion for the case when $\fg_2\ne 0$ in Section \ref{subsec:failasp}. 

Finally, we turn to examples of 1-shifted Lie bialgebras coming from loop groups, and combine our result with factorization structure, as in \cite{etingof1998quantizationliebialgebrasiii}. For $\fg$ a simple Lie algebra, and a difference-dependent tensor
\be
r(t_1-t_2)=\frac{C}{t_1-t_2}+g(t_1-t_2), \qquad g(t_1-t_2)\in \fg\lbb t_1\rbb\otimes \fg \lbb t_2\rbb
\ee
satisfying classical generalized Yang-Baxter equation (see Proposition \ref{Prop:splitr}), we produce a 1-shifted Lie bialgebra structure on $\fd (\CO)$ where $\fd=T^*[-1]\fg$ and $\CO=\C\lbb t\rbb$, as well as the associated quantization $U_\hbar (\fd (\CO))$. This DG algebra carries a differential $T=\pd_t$, with which one can define translation automorphism $\tau_z=e^{zT}$ for a formal variable $z$. By translating and re-expanding the 1-shifted $r$-matrix $\mathbf{r}(t_1-t_2)$, we obtain a tensor
\be
\mathbf{r}(t_1+z-t_2)\in \fd (\CO)\otimes \fd (\CO)\lbb z^\pm\rbb,
\ee
which we call the \textit{1-shifted meromorphic $r$-matrix}. This tensor allows us to prove the following result. 

\begin{Thm}[Section \ref{subsec:1shiftedmeror}]\label{Thm:1shiftedmerointro}
    Let $M^i$ be finitely-generated smooth modules of $U_\hbar (\fd (\CO))$ flat over $\Ch$ (which will be called FSF modules), and $z_i$ formal coordinates, there exists a DG module over $\C\lbb z_i\rbb [(z_i-z_j)^{-1}]$ of the form
    \be
\bigotimes \{M^i, z_i\}=(\bigotimes_i M^i)\lbb z_i\rbb[(z_i-z_j)^{-1}]
    \ee
    whose differential is given by
    \be
d_{\mathbf{r}}(z):= \sum d_{M^i}-2\hbar \sum_{i<j} \mathbf{r}^{ij}(t_1+z-t_2)\cdot -,
    \ee
    and whose action of $U_\hbar (\fd (\CO))$ is given by $\prod_i \tau_{z_i} \Delta^n$, where $\Delta$ is the symmetric coproduct of $U(\fd (\CO))$. 
\end{Thm}

The translated 1-shifted $r$-matrix $\mathbf{r}(t_1+z-t_2)$ serves the same role as definining a MC element on the tensor product of modules, inducing the $z$-dependent tensor products. We show in Section \ref{subsec:affineKacr} that when $r$ is assumed to be skew-symmetric, then one can add a level structure to $U_\hbar (\fd (\CO))$ to obtain a further deformation $U_\hbar^k (\fd (\CO))$ for $k\in \C$, and Theorem \ref{Thm:1shiftedmerointro} with the same 1-shifted meromorphic $r$-matrix holds for $U_\hbar^k (\fd (\CO))$. Moreover, in Section \ref{subsec:summarymero}, we comment that for any Lie algebra $\fg$ and $\fd=T^*[-1]\fg$, one can apply the proof of Theorem \ref{Thm:1shiftedmerointro} to the 1-shifted Lie bialgebra structure on $\fd (\CO)$ induced by the 1-shifted analogue of Yang's $r$-matrix. We call the associated DG algebras \textit{DG 1-shfited Yangians}, and denote them by  ${}_1\!Y_\hbar^k (\fd)$.\footnote{$k$ is present when there is an ordinary bi-linear form on $\fd$ to define the level. The subscript $1$ indicates that it is 1-shifted Yangian instead of the ordinary Yangian.} The natural physical interpretation of these structures is given in Section \ref{subsec:HTQFT}. 

\begin{Rem}
   The reason we call this algebra ``DG 1-shfited Yangian" is because the term ``shifted Yangians" already has a specific meaning. 
\end{Rem}

\subsection{Physical interpretation: 2d bulk-boundary TQFTs}\label{subsec:TQFT}

Extended defects play a central role in the study of topological quantum field theories (TQFT). For example, line operators in a 3 dimensional TQFT organize themselves into a braided tensor category. The prototypical example is Wilson lines in Chern-Simons theories \cite{witten1989quantum}. In finite semi-simple setting, the braided tensor category is enough to determine the state spaces associated to 2 manifolds and partition functions on 3-manifolds, via the Reshtikin-Turaev construction \cite{reshetikhin1991invariants}. 

The aforementioned Chern-Simons theory is also closely related to quantizations of Lie bialgebra structures. Roughly speaking, a Lie algebra $\fg$ with a non-degenerate bilinear pairing gives rise to a 3 dimensional Chern-Simons theory. Perturbatively, the category of line operators of this theory admits a description as $U(\fg)\Mod$, whose braided tensor structure comes from the quasi-Hopf construction of Drinfeld \cite{drinfeld1991quasi}. Quantum groups as defined by \cite{drinfeld1986quantum} are associated to monoidal fiber functors of this category. The idea of \cite{etingof1996quantization} is that a pair of transverse Lagrangian subalgebras of $\fg$ give rise to a monoidal fiber functor for Drinfeld's category and leads to quantizations of Lie bialgebras. Recently, in the work of the first author with T. Dimofte \cite{DNtanaka}, Etingof-Kazhdan's construction is given a physical interpretation as constructing transverse boundary conditions of the Chern-Simons theory from transverse Lagrangian subalgebras of $\fg$. 

The structure of 1-shifted Lie bialgebras and their doubles can be given a similar physical interpretation, this time as a bulk-boundary 2d TQFT. Indeed, it is noted in \cite{sBV} that a 1-shifted Lie bialgebra structure on $\fh$ corresponds to a Poisson structure on $\mathrm{CE}^*(\fh)$,  the Chevalley-Eilenberg cochain complex of $\fh$. In a similar vein, the 1-shifted metric on $\fg$ corresponds to a 1-shifted symplectic form on $\mathrm{CE}^*(\fg)$, in which $\mathrm{CE}^*(\fh)$ is a Lagrangian. This is the classical set-up of a bulk-boundary 2d TQFT. The relavent 2d TQFT should be a deformation of the so-called 2d topological BF theory, or a B twist of a 2d $\CN=(2,2)$ super Yang-Mills theory \cite[12.1.3]{elliott2022taxonomy}. Very roughly, any Lie algebra $\fh$ gives rise to such a 2d BF theory, and a 1-shifted Lie bialgebra structure gives rise to a compatible deformation of this theory by a potential of the form
\be
W=\hbar \sum_a \delta (B_a)A^a, \qquad \delta: \fh\to \fh\otimes \fh [-1]. 
\ee
Classically, the Lie algebra $\fh$ and $\fh^*[-1]$ should give rise to two boundary conditions that are transverse to each other. 

We conjecture that our proposed monoidal category $(U(\fg)\lbb\hbar\rbb, \hbar^2 W)\Mod$ is the category of interfaces of this 2d TQFT. When trying to quantize the boundary conditions defined by $\fh_\pm$, we found a potential obstruction in $\fg_2$, which leads to our assumption $\fg_2=0$. With this assumption, we conjecture that the DG algebras $U_\hbar (\fh)$ and $U_\hbar (\fh^*[-1])$ are local operators on the boundary conditions specified by $\fh$ and $\fh^*[-1]$. From this perspective, the monoidal actions of Theorem \ref{Thm:quantizeintro} is natural: interfaces form a monoidal category that naturally acts on the category of boundary conditions. 

In \cite{pantev2013shifted, calaque2017shifted}, the authors defined quantization of shifted Poisson algebras, based on formality of $\E_n$ operads. We expect that one can compare their quantization with the one we proposed here, but it is beyond the scope of the current paper. 

\subsection{3d HT theory and DG-shifted Yangians}\label{subsec:HTQFT}

One major motivation of this work, especially the formulation of Theorem \ref{Thm:1shiftedmerointro}, stems from the authors' collaboration with T. Dimofte \cite{lineope}, studying the algebraic structure of line operators in 3-dimensional holomorphic-topological (HT) theories. 

3-dimensional holomorphic-topological theories are quantum field theories that are holomorphic in 2 spacetime dimensions and topological in the remaining dimension. Locally, it requires the geometry of the spacetime to be of the form $\C\times \R$. Recently, much work has been devoted to the study of algebraic and geometric aspects of these theories, for example, \cite{aganagic2017topologicalchernsimonsmattertheories, dimofte2018dual, costello2023boundary, jockers20203d, ferrari2024boundary} and many more. 

In \cite{lineope}, we studied the HT twist of a 3d $\CN=2$ supersymmetric gauge theory. A 3d $\CN=2$ gauge theory is defined by a gauge group $G$ and a representation $V$ of $G$, and can be deformed by further including a superpotential. This theory admits a holomorphic-topological twist, which is the main focus of \textit{loc.cit}. We will only consider the HT twist of gauge theories with no superpotential in the following discussions. We particularly focused on the category of line operators $\CL$. Objects in this category are line operators along $\R$, and morphisms are given by local operators at the junction between two line operators. By inserting line operators at various points of $\C$, one expects that $\CL$ has the structure of a \textit{chiral} category. Our work was born out of an effort to try to understand the chiral structure as explicitly as possible. 

The meaning of being ``explicit" could depend on one's taste or on how one accesses this category $\CL$ mathematically. For example, one could access $\CL$ using boundary chiral algebras of \cite{costello2023boundary}, in a way much similar to Witten's approach to line operators of topological Chern-Simons theories \cite{witten1989quantum}. In this case, one defines
\be
\CL:= \CV\Mod
\ee
for the boundary vertex algebra $\CV$ on a holomorphic boundary condition. Under this identification, the chiral structure of $\CL$ should come from the chiral structure of $\CV$. In this sense, the chiral structure is simply a coherent sheaf of categories
$\CL_{\Sigma^n}$ over products of a complex curve $\Sigma$ satisfying the factorization properties of Beilinson-Drinfeld \cite{beilinson2004chiral}. For example, on $\Sigma^2$, this means that there are equivalences of categories
\be
\CL_{\Sigma^2}\Big\vert_{\Delta}\simeq \CL_\Sigma,\qquad \CL_{\Sigma^2}\Big\vert_{\Sigma^2\setminus \Delta}\simeq  \CL_\Sigma\boxtimes \CL_\Sigma \Big\vert_{\Sigma^2\setminus \Delta}. 
\ee
In some cases, such as loop spaces \cite{kapranov2004vertex} and VOAs \cite{beilinson2004chiral}, the chiral structure can be written very explicitly. However, this structure does not answer what happens when we consider the operator product expansion (OPE)
\be
\lim_{z\to w} L_1(z)L_2(w)=?, \qquad L_1, L_2\in \CL, z, w\in \Sigma. 
\ee
This is because unlike chiral algebras, there is no six-functor formalism or a notion of short exact sequences of categories, which is important to go from chiral algebras to vertex algebras and OPE. 

In \cite{lineope}, we derived this OPE via calculations of Feymann diagrams, using the idea of Koszul duality of Costello-Paquette \cite{costello2021twisted}. What we found in the end can be summarized as follows. There exists a DG algebra $A^!$ (signifying that it is the Koszul dual of the algebra of local operators $A$), with which we define
\be
\CL:= A^!\Mod. 
\ee
The algebra $A^!$ for the case of pure gauge theories was also derived in \cite{Paquette:2021cij}. We found that the OPE structure is governed by the following data.

\begin{enumerate}

    \item A Maurer-Cartan (MC) element $\alpha\in A^!\otimes A^!\lpp z^{-1}\rpp$ for a formal variable $z$. 

    \item An algebra homomorphism $\Delta: A^!\to A^!\otimes_\alpha A^!\lpp z^{-1}\rpp$, where the RHS is twisted by the MC element $\alpha$.

    \item Weak associativity and commutativity in terms of expansion of formal series. 
    
\end{enumerate}

For a pure gauge theory, the underlying algebra of $A^!$ is simply $U(T^*[-1]\fg[t])$, and the algebra homomorphism $\Delta$ turns out to be the symmetric coproduct. The differential is computed from Feymann diagrams. The element $\alpha$ takes the form:
\be
\alpha= \frac{x_a\otimes \epsilon^a-\epsilon^a\otimes x_a}{t_1+z-t_2},
\ee
and can be interpreted as the propagator of the theory. The above structures allow us to start with two objects in $\CL$, and produce a tensor product, which is an object in $\CL$ that depends meromorphically on a formal variable $z$. 

Two immediate questions arise. The first question is how this structure compares with the chiral structure of $\CV\Mod$. We have provided a conjectural answer to this in \cite{lineope}. The algebra $A^!$ is Koszul dual to the perturbative boundary VOA, which is simply the affine Kac-Moody algebra at some level. The formation of OPE is directly comparable to the fusion product algorithm of Kazhdan-Lusztig \cite{kazhdan1993tensor} and Gaberdiel \cite{gaberdiel1994fusion}. In fact, one can say that the algebra $A^!$ is the endomorphism of the fiber functor given by taking conformal block on $\mathbb P^1$ with a vacuum at $\infty$. Under this identification, the above tensor structure underlies the relation between OPE and conformal blocks. 

The second question is more fundamental, namely, whether there is a systematic explanation of the above structure. It was rather miraculous that $\alpha$ would satisfy the MC equation, let alone the fact that deforming the differential by $\alpha$ makes $\Delta$ a map of DG algebras. Of course, these must be true if one believes in the locality of the HT theory. We verified these statements in \cite{lineope} by direct calculations. 

In this paper, we explain this structure in terms of 1-shifted Lie bialgebras and their quantizations. In fact, we show that this phenomenon, namely that a monoidal structure on $\CA \Mod$ can arise from a coproduct from $\CA$ to $\CA\otimes_\alpha \CA$, is quite universal. Any 1-shifted Lie bialgebra and its double provide an example. The MC element is nothing but the 1-shifted $r$-matrix of the double. The algebra $A^!$ above is precisely the quantization of the 1-shifted Lie bialgebra $T^*[-1]\fg [t]$ (whose 1-shifted Lie bialgebra structure is induced by the vacuum at $\infty$), and the meromorphic MC element comes from re-expansion of the 1-shifted $r$-matrix after translation by $z$. This sufficiently justifies our calling $A^!$ the DG 1-shifted Yangian in \cite{lineope}, since it is the quantization of 1-shifted Lie bialgebra structure associated to Yang's $r$-matrix. 

In this sense, we actually achieved more. We show in Section \ref{sec:meromorphic} that there are different 1-shifted Lie bialgebra structures on $(T^*[-1]\fg)(\CO)$ induced by a classical difference-dependent $r$-matrix. By the work of Belavin and Drinfeld \cite{BDrmatrix}, such $r$-matrices arise from sheaf of Lie algebras over curves of genus $\leq 1$. Most notable examples are the trignometric and elliptic solutions. We expect that the meromorphic tensor products we obtained in Section \ref{sec:meromorphic} from these $r$-matrices can be extended to sheaves over corresponding curves, and are related to conformal blocks on those curves.

\subsection{Structure of the paper and future directions}

\subsubsection*{Structure of the paper}

\begin{itemize}

    \item In Section \ref{sec:shiftedlba}, we propose the definition of 1-shifted Lie bialgebras and discuss their properties. 

    \item In Section \ref{sec:quantdouble}, we propose a quantization for a 1-shifted metric Lie algebra, and for 1-shifted Lie bialgebras under the assumption \ref{Asp:g2=0}. 

    \item In Section \ref{sec:meromorphic}, we consider 1-shifted Lie bialgebra structures associated with loop groups, and obtain DG 1-shifted Yangians. We construct meromorphic tensor products from modules of these DG algebras.

\end{itemize}

\subsubsection*{Future directions}

\begin{itemize}

    \item The calculation of \cite{lineope} suggests that the statements in this work should be true for $L_\infty$ algebras. We would like to consider such generalizations. 

    \item The solutions $r(t_1-t_2)$ we used in this work are related to the geometry of complex curves of genus $\leq 1$. For higher genus curves, one can construct solutions $r(t_1, t_2, \lambda)$ that satisfy \textit{dynamical} Yang-Baxter equations \cite{abedin2024rmatrixstructurehitchinsystems}. In \cite{abedin2024quantumgroupoids}, a skew-symmetrization of this dynamical $r$-matrix was quantized.  It would be interesting to explore a similar generalization of the 1-shifted Lie bialgebra structure to Lie algebroids. If this is possible, the meromorphic tensor products for these dynamical $r$-matrices should be related to conformal blocks over higher genus curves. 

    \item It was hinted in \cite{DNtanaka} that relative double construction of Hopf algebras is related to transverse interfaces of 3d TQFTs as supposed to boundary conditions. Following this vein, it would be interesting to develop a relative notion of 1-shifted double. 

    \item Since all the meromorphic tensor products arise from the same HT theory, we expect that over products of $\mathbb D^\times$, different $r$-matrices result in the same meromorphic tensor structure, and must be related to each other via some twist. We would like to consider how they can be made twist-equivalent. 

    \item The magic of the 1-shifted case is that one can view $\fh_-$ as generators of the Chevalley-Eilenberg complex of $\fh_+$. It would be interesting to study $n$-shifted Lie bialgebras and their quantizations. These structures should have natural appearances in TQFTs of various dimensions, and should be related to the quantization of \cite{calaque2017shifted} by considering the induced structure on the Chevalley-Eilenberg complex.

\end{itemize}

\subsection{Conventions}\label{subsec:conventions}

\subsubsection*{Convetion on gradings}

In this paper, all Lie algebras and algebras are objects in the category of $\mathbb Z$-graded (or in full generality, DG) vector spaces. Since computations are very sensitive to this grading, we fix our convention as follows. 

Given an $\Z$-graded (we will simply call this graded, and mostly omit this since this is always part of the definition) vector space $V$, we denote by $F$ the grading map, namely $F: V\to V$ such that $F\vert_{V_n}=n\mathrm{Id}_{V_n}$. For $v\in V$, we will denote by $|v|$ the grading under $F$. Let $V$ and $W$ be to graded vector spaces, then the braiding isomorphism
\be
\sigma: V\otimes W\to W\otimes V
\ee
is given by $(-1)^{F\otimes F}\tau$ where $\tau(v\otimes w)=w\otimes v$. Namely $(v\otimes w)^{op}=(-1)^{|v||w|}w\otimes v$. Given two graded maps $f: V_1\to W_1$ and $g: V_2\to W_2$ of degree $|f|, |g|$, the tensor product $f\otimes g$ is the map such that
\be
f\otimes g(v_1\otimes v_2)=(-1)^{|g||v_1|} f(v_1)\otimes g(v_2). 
\ee

A commutative algebra $A$ is an algebra such that $m\circ \sigma=m$ where $m: A\otimes A\to A$ is multiplication. In terms of elements in $A$, this amounts to
\be
a\cdot b=(-1)^{|a||b|}b\cdot a.
\ee 

Let $\fg$ be a graded Lie algebra, namely it is a graded vector space together with a grading-perserving bracket $[-,-]: \fg\otimes \fg\to \fg$ satisfying graded anti-commutativity and Jacobi identity. When $\fg$ is graded finite-dimensional, this is equivalent to a map:
\be
\delta: \fg^*[-1]\to \mathrm{Sym}^2(\fg^*[-1])[1]
\ee
that extends to a square-zero differential of the commutative algebra $\mathrm{Sym}^*(\fg^*[-1])$ (the Chevalley-Eilenberg differential). If we choose a set of graded-basis $\{x_a\}$ of $\fg$ under which the structure constant is $f_{ab}^c$, and let $\{\epsilon^a\}$ be the corresponding dual elements in $\fg^*[-1]$ (namely $\langle \epsilon^a, x_b\rangle=\delta^a_b$), then $\delta$ is given by
\be
\delta (\epsilon^a)=\sum_{b\ne c}(-1)^{|b|(|c|+1)} f_{bc}^a\epsilon^b\otimes \epsilon^c. 
\ee
This can also be defined by
\be
\langle \delta(\epsilon^a), x_b\otimes x_c\rangle=\langle \epsilon^a, [x_b, x_c]\rangle. 
\ee

If $\fg$ is a Lie algebra, then the opposite $\fg^{op}$ is defined so that the bracket is given by $[-,-]\circ \sigma$. In other words, given two elements $X,Y\in \fg$, the opposite Lie bracket is given by
\be
[X, Y]_{\fg^{op}}=(-1)^{|X||Y|}[Y, X]_{\fg}. 
\ee
Define $S:\fg \to \fg$ by $S(X)=-X$, then it is naturally a Lie algebra homomorphism $S:\fg\to \fg^{op}$. It therefore induces a map of algebras $S: U(\fg)\to U(\fg^{op})=U(\fg)^{op}$ such that $S^2=1$. 

Let $V$ be a left $\fg$ module, then $V$ can be given a right $\fg$ module structure via $S$. We can thus make $V^*$ a left $\fg$ module again. Explicitly, let $X\in \fg, f\in V^*$ and $v\in V$, we have
\be
X\cdot f (v)=f(v\cdot X)=-(-1)^{|v||X|}f(X\cdot v). 
\ee
This gives an action of $\fg$ on $\fg^*[-1]$, which is given on the generators by
\be
x_a\cdot \epsilon^b=f_{ca}^b\epsilon^c. 
\ee
A graded map $f: V\to W$ of $\fg$-modules is a map such that
\be
Xf=(-1)^{|f||X|}fX.
\ee
For instance, the map $\delta$ is a degree 1 map of $\fg$ modules, in that it satisfies:
\be
\delta(X\cdot \epsilon)=(-1)^{|X|}X\cdot \delta(\epsilon). 
\ee

\subsubsection*{Invariant bilinear forms}

We would like to comment on the definition of non-degenerate bi-linear forms we impose on $\fg$, since this might cause confusions in comparison to \cite{sBV}. 

Let $\kappa$ be a degree $0$ non-degenerate invariant bi-linear form on $\fg$, then it satisfies:
\be
\kappa ([z, x], y)=\kappa (x, [y, z])\iff \kappa ([z, x], y)=(-1)^{|x|}\kappa (z, [x,y]). 
\ee
This induces an isomorphism of $\fg$ modules $\kappa: \fg\to \fg^*$. In this case, it is natural to require $\kappa$ to be symmetric, since
\be
\kappa ([z, x], y)=\kappa (x, [y, z])=(-1)^{|y|}\kappa ([x, y], z)=(-1)^{|y|}\kappa (y, [z, x]),
\ee
namely if $[-,-]$ is surjective, then $\kappa$ is forced to be symmetric. 

On the other hand, if we have some non-degenerate invariant bi-linear form $\kappa$ of degree $1$ in the sense that it induces an isomorphism $\kappa: \fg\to \fg^*[-1]$, then a choice can be made in whether it is symmetric or anti-symmetric. Indeed, from our convention, it is natural to assume
\be\label{eq:1xkappa}
\kappa ([x, y], z)=(-1)^{|x|}\kappa (y, [z, x]). 
\ee
If we apply an argument similar to the degree 0 case, we find an equality:
\be
\kappa ([x, y], z)=(-1)^{|x|+|y|+|z|+|z|(|x|+|y|)}\kappa (z, [x, y]).
\ee
However, using the fact that $\kappa$ is of degree $1$, we find that the above is non-zero only when $|x|+|y|+|z|=1$, in which case $|z|(|x|+|y|)=|z|(|z|+1)=0$, and we find
\be
\kappa ([x, y], z)=-\kappa (z, [x, y]).
\ee
Therefore, if $[-,-]$ is surjective, then $\kappa$ is anti-symmetric in the sense that $\kappa \circ \sigma=-\kappa$. This is the convention we will use, in contrast to the symmetric form of \cite{sBV}. We note that one can in fact build a symmetric form $\wt\kappa$ from 
$\kappa$ by defining
\be
\wt \kappa(x, y)=(-1)^{|x|}\kappa (x, y). 
\ee
Ultimately, the conventions we use and those of \cite{sBV} are equivalent via this re-definition. However, we think that it is more natural to work with the anti-symmetric bilinear form $\kappa$, since the 1-shifted $r$-matrix is the tensor associated to $\kappa$ rather than $\wt\kappa$. 

\subsection*{Acknowledgements}

We are particularly grateful to Tudor Dimofte for collaborations on projects that lead to this work, and would like to thank as well Kevin Costello and Davide Gaiotto for useful discussions related to this topic. W.N. is also very grateful to Raschid Abedin for explaining the relation between Lie algebra splittings of $\fg (\CK)$ and classical $r$-matrices, and to Lukas M\"uller for providing many helpful comments. Finally, W.N. is grateful to his friend Don Manuel for many encouragements. W.N.'s research is supported by Perimeter Institute for Theoretical Physics. Research at Perimeter Institute is supported in part by the Government of Canada through the Department of Innovation, Science and Economic Development Canada and by the Province of Ontario through the Ministry of Colleges and Universities.

\section{1-shifted Lie bi-algebras}\label{sec:shiftedlba}

\subsection{1-shifted Metric Lie algebras and Manin-triples}

Let $\fg$ be a finite-dimensional graded Lie algebra. We say that $\fg$ is a \textit{1-shifted} metric Lie algebra if it comes equipped with a degree $1$ anti-symmetric invariant non-degenerate bi-linear form $\kappa$ of degree $1$, namely a map:
\be\label{eq:kappa} 
\kappa: \fg\otimes (\fg[1])\to \C,
\ee
such that for all $x,y,z\in \fg$, the following hold:
\be
\kappa (x\otimes y)=-(-1)^{|x||y|}\kappa (y\otimes x), \qquad \kappa ([x, y]\otimes z)=(-1)^{|x|}\kappa (y\otimes [z, x])=0.
\ee
Note that since $\kappa (x\otimes y)\ne 0$ only if $|x|=|y|+1$ (as $\kappa$ is in degree $1$), the above can be re-written as
\be\label{eq:kappacommute}
\kappa (x\otimes y)=-\kappa (y\otimes x), \qquad \kappa ([y, x]\otimes z)=(-1)^{|x|} \kappa (y\otimes [x, z]), \qquad \forall x,y, z \in \fg.
\ee 
Moreover, non-degeneracy means that $\kappa$ induces an isomorphism $\fg\cong \fg^*[-1]$ as $\fg$ modules, sending $y$ to $\kappa (y\otimes -)$. 

An \textit{isotropic} Lie sub-algebra of a 1-shifted matrix Lie algebra $\fg$ is a Lie sub-algebra $\fh\subseteq \fg$ such that $\kappa (\fh, \fh)=0$. Or in other words, if we define $\fh^\perp$ to be the subspace of $\fg$ that pairs trivially with $\fh$ under $\kappa$, then $\fh$ is isotropic iff $\fh^\perp\supseteq \fh$. A \textit{co-isotropic} Lie sub-algebra of $\fg$ is a Lie sub-algebra $\fh$ such that $\fh^\perp\subseteq \fh$.   A Langrangian sub-algebra of $\fg$ is a Lie subalgebra $\fh$ such that $\fh=\fh^\perp$. 

\begin{Def}

A \textbf{1-shifted Manin-triple} is a triple $(\fg, \fh_+, \fh_-)$, where $\fg$ is a 1-shifted metric Lie algebra and $\fh_\pm$ are Lagrangian Lie subalgebras such that $\fg=\fh_+\oplus \fh_-$. We call $\{\fh_\pm\}$ transverse Lagrangian Lie subalgebras of $\fg$. 

\end{Def}

It is clear that if one is given a 1-shifted Manin-triple $(\fg, \fh_+, \fh_-)$, then there are natural identifications of graded vector spaces:
\be
\fh_+=\fh_-^*[-1],\qquad \fh_-=\fh_+^*[-1].
\ee
In particular, the second identification means that we have a map
\be
\delta: \fh_+\to \mathrm{Sym}^2(\fh_+)[1],
\ee
representing the Lie bracket of $\fh_-$. In other words,
\be
\kappa\otimes \kappa (\delta (X))(Y\otimes Y')=\kappa (X)([Y,Y']), \qquad X\in \fh_+, Y,Y'\in \fh_-. 
\ee
Moreover, using the decomposition $\fg=\fh_+\oplus \fh_-$, we see that there is a left action of $\fh_+$ on $\fh_-$ which we denote by $\rhd$, and a right action of $\fh_-$ on $\fh_+$, which we denote by $\lhd$, such that
\be
[X,Y]=X\rhd Y+X\lhd Y,\qquad X\in \fh_+, Y\in \fh_-.
\ee

\begin{Lem}\label{Lem:kappadual}
    Under the equivalence $\fh_-=\fh_+^*[-1]$ induced by $\kappa$, the action $\rhd$ coincides with the co-adjoint action of $\fh_+$ on $\fh_+^*[-1]$. 
    
\end{Lem}

\begin{proof}
We need to show that
\be
X\cdot \kappa (Y)=(-1)^{|X|} \kappa (X\rhd Y).
\ee
Take an element $X'\in \fh_+$, pairing both sides with it, the LHS is equal to
\be
\kappa (Y) ([X', X])=\kappa (Y, [X', X]),
\ee
whereas the RHS is equal to (since $\fh_+$ is Lagrangian)
\be
(-1)^{|X|}\kappa ([X, Y], X')=(-1)^{|X|}(-1)^{|X|}\kappa (Y, [X', X]).
\ee
This completes the proof. 

\end{proof}

\begin{Prop}\label{Prop:delta}
The bracket $\delta$ satisfies the following two equations:
\be
\delta \otimes 1(\delta)+1\otimes \delta (\delta)=0\in \mathrm{Sym}^3(\fh), \qquad \delta([X,Y])=[\delta(X), \Delta(Y)]+(-1)^{|X|}[\Delta(X), \delta(Y)].
\ee
In the above $\Delta$ is the symmetric co-product of $\fh_+$, and conjugation with $\delta(X)$ means its action on $\mathrm{Sym}^2(\fh_+)$.

\end{Prop}

\begin{proof}

The first equation is simply the Jacobi identity of $\fh_-$, and in fact is saying that $ \mathrm{Sym}^*(\fh_+)$ can be naturally turned into the Chevalley-Eilenberg complex of $\fh_-$. We show that the second equation follows from the Jacobi identity of $\fg$. Indeed, let $Z, W\in \fh_-$, and $X\in \fh_+$. We have the following Jacobi identity:
\be
(-1)^{|Z||X|}[[Z,W], X]+(-1)^{|Z||W|}[[W, X], Z]+(-1)^{|W||X|}[[X, Z], W]=0
\ee
If we quotient by $\fh_+$ onto the sub-space $\fh_-$, we obtain the following identity:
\be\label{eq:ZXY}
X\rhd [Z, W]=[X\rhd Z, W]+(X\lhd Z)\rhd W-(-1)^{|Z||W|} \lp [X\rhd W, Z]+(X\lhd W)\rhd Z\rp.
\ee
Let $X, Y\in \fh_+$, and $Z, W\in\fh_-$, then
\be
(\kappa\otimes \kappa) \delta ([X, Y]) (Z\otimes W)=\kappa ([X, Y], [Z, W])=(-1)^{|Y|}\kappa (X, [Y, [Z,W]])=(-1)^{|Y|}\kappa (X, Y\rhd [Z,W]]).
\ee
Using equation \eqref{eq:ZXY}, the RHS has two terms. The first one is
\be
(-1)^{|Y|}\lp \kappa (X, [Y\rhd Z, W])+ (-1)^{|Y||Z|}\kappa (X,[Z, Y\rhd W])\rp=[\delta (X), \Delta (Y)](Z\otimes W).
\ee
Here the last equation follows from equation \eqref{eq:kappacommute}. The second term is
\be
(-1)^{|Y|}\lp \kappa (X, (Y\lhd Z)\rhd W)-(-1)^{|Z||W|} \kappa (X, (Y\lhd W)\rhd Z) \rp.
\ee
We can change this into something similar to the first term in the following way:
\be
\begin{aligned}
    \kappa (X, (Y\lhd Z)\rhd W)&= \kappa (X, [(Y\lhd Z), W])=(-1)^{|W|} \kappa ([W, X], Y\lhd Z)\\ &=(-1)^{|W|} \kappa (W\lhd X, [Y, Z])=-(-1)^{|W|} \kappa ([Y, Z], W\lhd X)\\ &=(-1)^{|W|}(-1)^{|W||X|}\kappa ([Y, Z], X\rhd W)\\ &= (-1)^{|W|}(-1)^{|W||X|} (-1)^{|Z|}\kappa (Y, [Z, X\rhd W])
\end{aligned}
\ee
This can be similarly done for $\kappa (X,  (Y\lhd W)\rhd Z)$. After fixing all the signs using $|X|+|Y|+|Z|+|W|=1$, we find in the end that the second term is equal to
\be
-(-1)^{|X||Y|}(-1)^{|X|} \lp \kappa (Y, [X\rhd Z, W])+ (-1)^{|X||Z|}\kappa (Y,[Z, X\rhd W])\rp
\ee
and this is readily equal to
\be
-(-1)^{|X||Y|} [\delta (Y), \Delta (X)](Z\otimes W)=(-1)^{|X|} [\Delta (X), \delta (Y)] (Z\otimes W). 
\ee
This completes the proof. 

\end{proof}

Symmetrically, the same is true for $\fh_-$. In summary, from a 1-shifted Manin-triple $(\fg, \fh_+,\fh_-)$, we obtain two Lie subalgebras $\fh_\pm$ with two anti-symmetric cobrackets
\be
\delta_\pm: \fh_\pm\longrightarrow \fh_\pm \otimes \fh_\pm [1]
\ee
satisfying the cocycle conditions of Proposition \ref{Prop:delta}. The structures on the two Lagrangian subalgebras are not independent: taking 1-shifted dual swaps $\fh_+$ and $\fh_-$, brackets with cobrackets. 

\subsection{1-shifted Lie bi-algebra and its classical double}

Given Proposition \ref{Prop:delta}, we see that a 1-shifted Manin-triple $(\fg, \fh_+,\fh_-)$ gives rise to a tensor
\be
\delta: \fh_+\to \mathrm{Sym}^2(\fh_+)[1]
\ee
satisfying certain cocycle conditions. This is extremely similar to the classical Lie bi-algebra structure, albeit with a degree shift. This motivates the following definition. 

\begin{Def}\label{Def:LBAmain}
 
 A 1-shifted Lie bi-algebra structure on $\fh$ is a co-bracket
 \be
 \delta: \fh\to \mathrm{Sym}^2(\fh)[1]
 \ee
satisfying
\be\label{eq:shiftedbialg}
\delta \otimes 1(\delta)+1\otimes \delta (\delta)=0\in \mathrm{Sym}^3(\fh), \qquad \delta([X,Y])=[\delta(X), \Delta(Y)]+(-1)^{|X|}[\Delta(X), \delta(Y)],
\ee
for all $X, Y\in\fh$. 

\end{Def}

In the theory of classical Lie bi-algebras, there is an equivalence between Manin-triples and Lie bi-algebras. Proposition \ref{Prop:delta} shows how from a 1-shifted Manin-triple one obtains a 1-shifted Lie bi-algebra (in fact two of them). We now show that the converse is true. 

\begin{Thm}\label{Thm:Manindelta}

Given a 1-shifted Lie bi-algebra $(\fh, \delta)$, there exists a 1-shifted Manin triple $(\fg, \fh, \fh^*[-1])$ inducing the co-bracket $\delta$ on $\fh$. This gives a one-one correspondence between 1-shifted Lie bi-algebras and 1-shifted Manin-triples. 

\end{Thm}

\begin{proof}

Let $\fh^*[-1]$ be the dual Lie algebra defined by $\delta$, and let $\delta^*$ be the 1-shifted co-bracket on $\fh^*[-1]$ dual to the Lie bracket of $\fh$. Let us first show that $(\fh^*[-1], \delta^*)$ also defines a 1-shifted Lie bialgebra. 

The property that $\delta^*\otimes 1(\delta^*)+1\otimes \delta^*(\delta^*)=0$ again follows from the Jacobi identity of $\fh$. We show the second identity in equation \eqref{eq:shiftedbialg}. 

To this end, let us choose graded basis $\{x_a\}$ of $\fh$, and dual basis $\{\epsilon^a\}$ of $\fh^*[-1]$, so that $(x_a, \epsilon^b)=\delta_a^b$. Note that $|x_a|=|\epsilon^a|+1$. Let $f_{ab}^c$ the structure constant of $x_a$, and $g^{ab}_c$ the structure constants of $\epsilon^a$. Then
\be
\delta^* (\epsilon^a)=-\sum (-1)^{|x_b|(|x_c|+1)}f^a_{bc}\epsilon^b\otimes \epsilon^c, \qquad \delta (x_a)=\sum (-1)^{|x_c|(|x_b|+1)}g^{bc}_a x_b\otimes x_c
\ee
For simplicity, we denote by $|a|$ the parity of $x_a$, and the parity of $\epsilon^a$ will be $|a|+1$. The second identity of equation \eqref{eq:shiftedbialg} for $\fh$ is given by the following
\be\label{eq:deltaabc}
\begin{aligned}
  \sum_c  f_{ab}^c(-1)^{|e|(|d|+1)} g^{de}_c & =  \sum_c(-1)^{|c|(|d|+1)}g^{dc}_a f_{cb}^e + (-1)^{|e|(|c|+1)+|e||b|}g^{ce}_af_{cb}^d \\ &-(-1)^{|a||b|}\sum_c \lp  (-1)^{|c|(|d|+1)}g^{dc}_b f_{ca}^e +(-1)^{|e|(|c|+1)+|e||a|}g^{ce}_bf_{ca}^d\rp.
\end{aligned}
\ee
We show that this is equivalent to the corresponding identity for $\fh^*[-1]$, that is,
\be\label{eq:deltastarabc}
\begin{aligned}
  &-\sum_c  f_{ab}^c(-1)^{|a|(|b|+1)} g^{de}_c  =  -\sum_c(-1)^{|c|(|b|+1)+(|e|+1)(|b|+1)} f_{cb}^d g^{ce}_a  + (-1)^{|a|(|c|+1)}  g^{ce}_b f_{ac}^d \\ &-(-1)^{(|d|+1)(|e|+1)}\sum_c \lp  (-1)^{|a|(|c|+1)+(|d|+1)(|a|+1)}g^{dc}_b f_{ac}^e +(-1)^{|c|(|b|+1)}g^{dc}_a f_{cb}^e\rp\,.
\end{aligned}
\ee

To prove this, let us first multiply both sides of equation \eqref{eq:deltastarabc} by $(-1)^{|a|(|b|+1)+|e|(|d|+1)}$ in order for its LHS to match that of equation \eqref{eq:deltaabc}. It then suffices to show that the four terms on the RHS of both equations are equal term by term. Let us check this explicitly for the first term of equation \eqref{eq:deltastarabc}, and show that it corresponds to the second term of equation \eqref{eq:deltaabc}. The other terms can then be verified with similar calculations.

After multiplication by the power of $(-1)$ described above, the first term of \eqref{eq:deltastarabc} reads
\be
(-1)^{|e|(|d|+1)}(-1)^{|a|(|b|+1)} (-1)^{|c|(|b|+1)} f_{cb}^d (-1)^{(|e|+1)(|b|+1)}g^{ce}_a \epsilon^a\otimes \epsilon^b,
\ee
which is comparable to the second term of \eqref{eq:deltaabc}. We therefore need to prove the identity
\be\label{eq:term2delta*abc}
\sum_c (-1)^{|e|(|d|+1)}(-1)^{|a|(|b|+1)} (-1)^{|c|(|b|+1)} (-1)^{(|e|+1)(|b|+1)}f_{cb}^dg^{ce}_a= \sum_c(-1)^{|e|(|c|+1)+|e||b|}g^{ce}_af_{cb}^d.
\ee
To verify this, we note that the indices satisfy further constraints
\be
|c|+|b|=|d|, \qquad |c|+|e|=|a|+1,
\ee
and consequently 
\be
\begin{aligned}
    |e|(|d|+1)+(|a|+|c|+|e|+1)(|b|+1)&=|e|(|c|+|b|+1)+(2|a|+2)(|b|+1)\\ \text{(since powers are mod 2)}\ \ \ \  &= |e|(|c|+1)+|e||b|\, ,
\end{aligned}
\ee
which coincides with the sign on the RHS of equation \eqref{eq:term2delta*abc}. 

The equivalence between the other three terms of equation \eqref{eq:deltastarabc} and those of equation \eqref{eq:deltaabc} can be checked using similar steps.




To finish, we define a Lie bracket on $\fg:=\fh\oplus \fh^*[-1]$ by
\be
[X,Y]=X\rhd Y+X\lhd Y, \qquad X\in \fh, Y\in \fh^*[-1]
\ee 
Here $\rhd$ is the left action of $\fh$ on its shifted dual, multiplied by $(-1)^{|*|}$, and $\lhd$ is the right action of $\fh^*[-1]$ on its shifted dual, multiplied by $(-1)^{|*|}$. The multiplication of the parity is due to Lemma \ref{Lem:kappadual}. To show that this, together with the Lie bracket on $\fh, \fh^*[-1]$ makes $\fg$ into a graded Lie algebra, we simply need to check Jacobi identity. Let us start with $X, Y\in \fh^*[-1]$ and $Z\in\fh $, then the Jacobi identity of these three elements is equivalent to equation \eqref{eq:ZXY} together with the following equality:
\be\label{eq:XYZ}
Z\lhd [X, Y]=Z\lhd X\lhd Y-(-1)^{|X||Y|}Z\lhd Y\lhd X.
\ee
Equation \eqref{eq:ZXY} follows from the co-cycle condition of $\delta$, as the proof of Proposition \ref{Prop:delta} has shown, and equation \eqref{eq:XYZ} follows from the fact that the action of $\fh^*[-1]$ on its dual is a Lie-algebra action. The same is true if one has two elements in $\fh$ and one in $\fh^*[-1]$. This completes the proof.

\end{proof}

\begin{Def}

For a 1-shifted Lie bi-algebra  $(\fh, \delta)$, we call the Lie algebra $\fg$ supplemented by Theorem \ref{Thm:Manindelta} the 1-shifted double of $\fh$, and denote it by $D_1(\fh, \delta)$, or simply $D_1(\fh)$ if no confusion can be caused. 

\end{Def}

\begin{Exp}\label{Exp:cotangent}

For any DG Lie algebra $\fh$, we can view this as a trivial 1-shifted Lie bi-algebra. The corresponding double in this case is simply $\fh\ltimes \fh^*[-1]$, which can also be denote by $T^*[-1]\fh$. Namely, $D_1(\fh, 0)=T^*[-1]\fh$. 

\end{Exp}

The double of an ordinary Lie bi-algebra is also a Lie bi-algebra, and moreover has a classical $r$-matrix such that the co-bracket on the double is induced from this classical $r$-matrix. Let us now consider $\fh$ a 1-shifted Lie bi-algebra, and $\fg:=D_1(\fh)$ its double. Let $\{x_a\}\subseteq\fh$ a set of graded basis and $\{\epsilon^a\}\subseteq \fh^*[-1]$ the corresponding dual basis, such that $\kappa( x_a, \epsilon^b)=\delta_a^b$. Define $\mathbf{r}\in \fg\otimes \fg$ by:
\be
\mathbf{r}=\sum_a x_a\otimes \epsilon^a\in \fh\otimes \fh^*[-1],
\ee
which is an element in $\fg\otimes \fg$ of degree $1$. This element satisfies
\be
1\otimes \kappa (\mathbf{r}, X)=-X, \qquad \kappa\otimes 1(\mathbf{r}, Y)=(-1)^{|Y|}Y, \qquad X\in \fh, Y\in \fh^*[-1]. 
\ee
Moreover $\kappa \otimes \kappa (\mathbf{r}, Y\otimes X)=\kappa ((-1)^{|Y|}Y, X)=\wt\kappa (Y, X)$ for $X\in \fh, Y\in \fh^*[-1]$.  

\begin{Prop}\label{Prop:coboundaryr}

Denote by $\delta_\fg$ the co-bracket $\delta_\fh\oplus -\delta_{\fh^*[-1]}$. Then:
\be
\delta_\fg=[-\mathbf{r}, \Delta].
\ee

\end{Prop}

\begin{proof}

Let us choose $X\in \fh$ and $Y, Z\in \fg$, and consider
\be
\kappa\otimes \kappa ([-\mathbf{r}, \Delta(X)], Y\otimes Z).
\ee
We need to show that this is nonzero only when $Y, Z\in \fh^*[-1]$, and is equal to $\kappa (X, [Y, Z])$. Now we have
\be
\kappa\otimes \kappa ([-\mathbf{r}, \Delta(X)], Y\otimes Z)=(-1)^{|X|}\kappa\otimes \kappa (-\mathbf{r},[\Delta (X), Y\otimes Z]),
\ee
and so if $Y\in \fh$, then the above is automatically zero. Let us assume now that $Z\in \fh$ and $Y\in \fh^*[-1]$. Then we have
\be
\begin{aligned}
    \kappa\otimes \kappa (-\mathbf{r},[\Delta (X), Y\otimes Z])&=\kappa\otimes \kappa (-\mathbf{r},[X, Y]\otimes Z)+(-1)^{|X||Y|} \kappa\otimes \kappa (-\mathbf{r},Y\otimes [X,Z])
    \\ &=-(-1)^{|X|+|Y|}\kappa ([X, Y], Z)-(-1)^{|X||Y|}(-1)^{|Y|}\kappa (Y, [X, Z])\\ &= (-1)^{|X|+|Y|}(-1)^{|X|+|X||Z|}\kappa (Y, [X, Z])-(-1)^{|X||Y|}(-1)^{|Y|}\kappa (Y, [X, Z])
\end{aligned}
\ee
Since the above is non-zero only when $|X|+|Y|+|Z|=1$, we find in the end that $|X|+|Y|+|X|+|X||Z|=|Y|+|X||Y|$, and the above is zero. Finally, if $Y, Z\in \fh^*[-1]$, then the first term above drops off, and we have
\be
\begin{aligned}
    \kappa\otimes \kappa (-\mathbf{r},[\Delta (X), Y\otimes Z])&= (-1)^{|X||Y|} \kappa\otimes \kappa (-\mathbf{r},Y\otimes [X,Z])\\ &=-(-1)^{|X||Y|}(-1)^{|Y|}\kappa (Y, [X, Z])\\ &=-(-1)^{|X||Y|}(-1)^{|Y|}(-1)^{|X|+|Y|}\kappa (X, [Z, Y])\\ &=(-1)^{|X||Y|}(-1)^{|Y|}(-1)^{|X|+|Y|}(-1)^{|Y||Z|}\kappa (X, [Y, Z])
\end{aligned}
\ee
Since this is non-zero only when $|X|+|Y|+|Z|=1$, we find that the above sign is equal to $(-1)^{|X|}$, as desired. For $\fh^*[-1]$, we can perform a similar argument, swapping $\mathbf{r}$ by $-\mathbf{r}^{21}$. The proof is complete. 

\end{proof}

As a consequence, the Lie algebra $\fg=D_1 (\fh)$ has the structure of a 1-shifted Lie bi-algebra induced by $\delta_\fg:=[-\mathbf{r}, \Delta]$. Indeed, from the above we see that this defines a map
\be
\delta: \fg\to \mathrm{Sym}^2(\fg)[1],
\ee
and clearly satisfies the second identity of equation \eqref{eq:shiftedbialg}. It also satisfies the first identity since it does on both $\fh$ and $\fh^*[-1]$ and $\fg=\fh\oplus \fh^*[-1]$. We will call this $\mathbf{r}$ the 1-shifted $r$-matrix of $\fh$, due to its similar role as the ordinary $r$-matrix. We now show that an analogue of classical Yang-Baxter equation is satisfied by $\mathbf{r}$. 

\begin{Prop}\label{Prop:1shiftedYB}
    This element $\mathbf{r}$ satisfies the following equalities.
\be\label{eq:classicaldr[rr]}
\delta_{\fg}\otimes 1 (\mathbf{r})=[\mathbf{r}^{13}, \mathbf{r}^{23}],\qquad 1\otimes \delta_{\fg}(\mathbf{r})=[\mathbf{r}^{12}, \mathbf{r}^{13}].
\ee
Consequently, the following classical Yang-Baxter equation holds
  \be\label{eq:1shiftedYB}
[\mathbf{r}^{12}, \mathbf{r}^{13}]+[\mathbf{r}^{12}, \mathbf{r}^{23}]+[\mathbf{r}^{13},\mathbf{r}^{23}]=0\in \fg\otimes \fg\otimes \fg. 
    \ee

\end{Prop}

\begin{proof}
      Let us start by showing that $\delta_{\fg}\otimes 1(\mathbf{r})=[\mathbf{r}^{13},\mathbf{r}^{23}]$. Let $X, Y\in \fh^*[-1]$, we have
      \be
\kappa\otimes 1 (\delta_{\fg}\otimes 1(\mathbf{r}), X\otimes Y)=\kappa\otimes 1 (\mathbf{r},[X, Y])=(-1)^{|X|+|Y|}[X, Y].
      \ee
      On the other hand, we have
      \be
\kappa\otimes 1 [\mathbf{r}^{13},\mathbf{r}^{23}] (X\otimes Y)=[(-1)^{|X|}X, (-1)^{|Y|}Y]=(-1)^{|X|+|Y|}[X, Y].
      \ee
      This shows that the two agrees. Equality $1\otimes \delta_{\fg}(\mathbf{r})=[\mathbf{r}^{12}, \mathbf{r}^{13}]$ follows from an identical calculation, and in fact follows from the cYBE of $\mathbf{r}$, which we prove now. 

 Using the definition of $\delta_\fg$, we have:
    \be
\delta_{\fg}\otimes 1 (\mathbf{r})=-[\mathbf{r}^{12}, \Delta\otimes 1(\mathbf{r})]= -[\mathbf{r}^{12}, \mathbf{r}^{13}]-[\mathbf{r}^{12}, \mathbf{r}^{23}].
    \ee
Consequently, equality $\delta_{\fg}\otimes 1 (\mathbf{r})=[\mathbf{r}^{13},\mathbf{r}^{23}]$ implies
\be
[\mathbf{r}^{12}, \mathbf{r}^{13}]+[\mathbf{r}^{12}, \mathbf{r}^{23}]+[\mathbf{r}^{13},\mathbf{r}^{23}]=0.
\ee
    
\end{proof}

\section{Quantization of 1-shifted Manin-triples}\label{sec:quantdouble}

Given a 1-shifted Manin-triple $(\fg, \fh_+,\fh_-)$, we have seen that there are 1-shifted Lie bialgera structures defined on $(\fg, \fh_\pm)$ such that the co-bracket is co-boundary for $\fg$ and induced by a classical $r$-matrix. In this section, we consider the problem of their quantization. 

We will first propose a quantization of the double $\fg$, in the sense that we construct a deformation of the symmetric monoidal category $U(\fg)\Mod$ over $\C\lbb\hbar\rbb$, which is canonical and non-symmetric for $\hbar\ne 0$. Our proposed quantization is in terms of a curved differential graded algebra (CDGA for short) $(U(\fg)\lbb\hbar\rbb,\hbar^2 W)$, rather than an ordinary DG algebra. From the proof, it will be clear that from a pair of transverse Lagrangians, one can twist this into a differential CDGA $(U(\fg)\lbb\hbar\rbb,d_{\mathbf r}, \hbar^2 c)$ where $c\in \fg_2$. We describe the monoidal structure by making the quantization into co-unital co-algebra objects in the category of CDGAs. The 1-shifted $r$-matrix will enter into the definition of the coproduct. 

We then turn to the quantization of $\fh_\pm$. Unfortunately we don't know how to construct them in full generality. We will impose the dramatically-simplifying assumption \ref{Asp:g2=0}, which forces the curvature $c$ above to vanish. We show that one can find DG algebras $U_\hbar (\fh_\pm)\subseteq U_\hbar (\fg)$, such that $U_\hbar (\fh_\pm)$ are Koszul dual to each other, and $U_\hbar (\fh_\pm)\Mod$ are monoidal module categories of $U_\hbar (\fg)$. Though restricting, this assumption is satisfied by examples we consider in Section \ref{sec:meromorphic}.

\begin{Rem}
    In all that follows, modules of $U(\fg)\lbb\hbar\rbb$ are always flat over $\Ch$. 
\end{Rem}

\subsection{Quantization of the double}\label{subsec:quantdouble}

\subsubsection{Recollection on curved differential graded algebras}

We start by recalling quickly the definition of a curved differential graded algebra (CDGA for short) and related concepts. For more details, we refer the readers to \cite{positselski2023differential}. 

 A \textbf{curved differential graded algebra} (CDGA) is a triple $(A, d, W)$ consisting of
\begin{enumerate}
    \item A graded algebra $A$.

    \item A degree $1$ differential $d$ on $A$.

    \item An element $W\in A_2$ such that $dW=0$ and $d^2=[W, -]$. 
\end{enumerate}
The element $W$ is usually called a curvature. We use the notation $W$ since in physics, such a curvature can come from including a potential function, usually denoted by $W$. One simple example of a CDGA is a graded algebra $A$ with a central element $W$ of degree $2$. A map between two CDGAs $(A, d_A, W_A)$ and $(B, d_B, W_B)$ is a pair $(f, \alpha)$ where
\begin{enumerate}
    \item $f:A\to B$ is a map of graded algebras, $\alpha\in B_1$. 

    \item $f(d_Aa)=d_B f(a)+[\alpha, f(a)]$ and $f(W_A)=W_B+d_B \alpha+\alpha^2$. 
    
\end{enumerate}
This map is called an isomorphism if $f$ is an isomorphism. Note that the definition of a morphism includes a MC element, which is important for us since the monoidal structure on the quantization is induced by maps of CDGAs, even if $W$ is zero. The composition of two morphisms $(f, \alpha)$ and $(g, \beta)$ is $(g\circ f, \beta+g(\alpha))$. The collection of CDGAs form a symmetric monoidal category. 

A module of a CDGA $(A, d, W)$ is a tuple $(M, d_M)$ such that
\begin{enumerate}
    \item $M$ is a graded module of $A$.

    \item $d_M$ is a map of degree $1$ such that $[d_M, a]=d(a)$ and $d_M^2= W$. 
\end{enumerate}
All modules of a CDGA form a DG category which we denote by $(A, d, W)\Mod$. A morphism $(f, \alpha)$ from $(A, d_A, W_A)$ to $(B, d_B, W_B)$ induces a functor
\be
(f, \alpha)_*: (B, d_B, W_B)\Mod\to (A, d_A, W_A)\Mod.
\ee
Explicitly, this functor maps an object $(M, d_M)$ to $(M_f, d_M+\alpha)$, where $M_f$ is a module of $A$ via restriction along $f$. Two isomorphic CDGAs have equivalent category of modules.

\subsubsection{Curvature on the double}

Let us now fix a 1-shifted metric Lie algebra $\fg$, and choose any pair of transverse Lagrangians $(\fh_\pm)$. There is always a canonical such choice given by
\be\label{eq:canLag}
\fh_+=\bigoplus_{n\leq 0}\fg_n, \qquad \fh_-=\bigoplus_{n\geq 1} \fg_n.
\ee
We will construct a central element $W$ of degree $2$ in $U(\fg)$ using such a choice, and show along the way that this element does not depend on the choice. 

Fixing such a choice, we obtain a co-bracket $\delta_{\mathbf r}: \fg\to \mathrm{Sym}^2(\fg)[1]$, given by $\delta_{\mathbf r}=[-\mathbf{r}, \Delta]$. Let us denote by $d_{\mathbf{r}}$ the map
\be
\btik
\fg\rar{\hbar \delta_\mathbf{r}} & \mathrm{Sym}^2(\fg)\lbb\hbar\rbb [1]\rar{\nabla} & U(\fg)\lbb\hbar\rbb [1]
\etik
\ee
where $\nabla$ is the multiplication map. From the second identify of equation \eqref{eq:shiftedbialg}, we conclude that $d_{\mathbf r}$ is a differential, and can be extended to the entire algebra $U(\fg)\lbb\hbar\rbb$. Moreover, it is inner and given by
\be
d_\mathbf{r}=[-\hbar \nabla\mathbf{r}, -], .
\ee
Since $\delta_\mathbf{r}$ is valued in the symmetric algebra, we can also conclude that $C=\nabla\mathbf{r}-\nabla\sigma \mathbf{r}$ is central. Let us denote by
\be
\Omega=\mathbf{r}-\sigma\mathbf{r}\in \fg\otimes \fg, \qquad \rho= -\frac{1}{2}\lp \nabla\mathbf{r}+\nabla\sigma \mathbf{r}\rp,
\ee
so that $d_\mathbf{r}=[\hbar \rho, -]$ and $C=\nabla \Omega$. 

\begin{Lem}
The tensor $\Omega$ is canonical (namely does not depend on the choice of $(\fh_\pm)$), and satisfies
\be
[\Omega, \Delta (x)]=0, \forall x\in \fg.
\ee

\end{Lem}

\begin{proof}

    The fact that $\Omega$ is canonical follows from the fact that $\Omega$ is the quadratic Casimir element associated to $\kappa$. Indeed, if we choose basis $\{x_a\}$ of $\fh_+$ with dual basis $\{\epsilon^a\}\in\fh_-$, then $\{x_a, \epsilon^b\}$ is a basis for $\fg$ with dual basis $\{\epsilon^a, -x_b\}$ under $\kappa$. The fact that $[\Omega, \Delta]=0$ follows from the fact that $\delta$ is valued in $\mathrm{Sym}^2(\fg)$, or that $\kappa$ is invariant.  
    
\end{proof}

Since $d_\mathbf{r}=[\hbar\rho, -]$, we see that
\be
d^2_\mathbf{r}=[\frac{\hbar^2}{2}[\rho, \rho], -].
\ee
Since $\rho\in \mathrm{Sym}^2(\fg)\subseteq U(\fg)$, the commutator $\frac{1}{2}[\rho, \rho]$ belongs to
\be
\frac{1}{2}[\rho, \rho]\in \mathrm{Sym}^3(\fg)\oplus \fg.
\ee
We can therefore write $\frac{1}{2}[\rho, \rho]=c_\mathbf{r}-W_\mathbf{r}$, for $c_\mathbf{r}\in \fg$ and $W_\mathbf{r}\in \mathrm{Sym}^3(\fg)$. 

\begin{Lem}\label{Lem:d2=c}
    The element $W_\mathbf{r}$ is central, and therefore
    \be
d^2_\mathbf{r}=[\hbar^2 c_\mathbf{r}, -].
    \ee
\end{Lem}

\begin{proof}
    This follows from the first identity of equation \eqref{eq:shiftedbialg}. Indeed, for any $X\in \fg$, we have
    \be
d^2_\mathbf{r}(X)=\hbar^2 [c_\mathbf{r}-W_\mathbf{r}, X]=\hbar^2[c_\mathbf{r}, X]-\hbar^2 [W_\mathbf{r}, X].
    \ee
    However, the first identity of equation \eqref{eq:shiftedbialg} implies that the image of this element after projecting to $\mathrm{Sym}^3(\fg)$ is zero, and so $[W_\mathbf{r}, X]=0$ for all $X$. 
    
\end{proof}

This $W_\mathbf{r}$ is the desired curvature, making $(U(\fg)\lbb\hbar\rbb,  \hbar^2 W_\mathbf{r})$ a CDGA. However, we haven't been able to show that $W_\mathbf{r}$ is canonical, in the sense that it doesn't depend on the choice of a pair of transverse Lagrangians. We will do so in the next section, together with the construction of the monoidal structure. 

\begin{Rem}\label{Rem:d2odd}

    We don't need to use the fact that $d$ is inner to derive that $d^2\fg\subseteq \fg\oplus \mathrm{Sym}^3(\fg)$. In fact, for any differential $d$ on $U(\fg)$ such that $d\fg\subseteq \mathrm{Sym}^2(\fg)$, it is always true that $d^2\fg\subseteq \fg\oplus \mathrm{Sym}^3(\fg)$. Namely, one can show that after total symmetrization, the $\mathrm{Sym}^2$ part is equal to zero. 
    
\end{Rem}

\subsubsection{Independence of Lagrangians and monoidal structure}

We have constructed a differential $d_\mathbf{r}$ on $U(\fg)\lbb\hbar\rbb$ and a central element $W_{\mathbf r}$ of degree $2$, as well as an element $c_{\mathbf r}\in \fg_2$. The following theorem is at the core of our construction. 

\begin{Thm}\label{Thm:dr=r2} 
    The following identity holds in $U(\fg)\lbb\hbar\rbb\Chtensor U(\fg)\lbb\hbar\rbb$. 
    \be\label{eq:dr=r2}
\hbar (d_\mathbf{r}\otimes 1+1\otimes d_\mathbf{r})(\mathbf{r})=\hbar^2[\mathbf{r}, \mathbf{r}], \qquad  \Delta d_\mathbf{r}=(d_\mathbf{r}\otimes 1+1\otimes d_\mathbf{r}-2\hbar [\mathbf{r}, -])\Delta,
    \ee 
    where $\Delta$ is the symmetric coproduct of $U(\fg)\lbb\hbar\rbb$. 
    
\end{Thm}

\begin{proof}
    The first equality of \eqref{eq:dr=r2} is a consequence of \eqref{eq:classicaldr[rr]}. Indeed, using these two equations, we can write
\be
\hbar d_\mathbf{r}\otimes 1(\mathbf{r})=\hbar^2 \nabla^{12}[\mathbf{r}^{13}, \mathbf{r}^{23}], \qquad \hbar 1\otimes d_\mathbf{r} (\mathbf{r})=\hbar^2\nabla^{23}[\mathbf{r}^{12}, \mathbf{r}^{13}].
\ee
Here $\nabla^{ij}$ means applying multiplication on the $i,j$ factor. Moreover, it is clear that
\be
[\mathbf{r}, \mathbf{r}]= \nabla^{12}[\mathbf{r}^{13}, \mathbf{r}^{23}]+ \nabla^{23}[\mathbf{r}^{12}, \mathbf{r}^{13}]. 
\ee
From these two equations one derives the first equality of \eqref{eq:dr=r2}. 

To prove the second equality, note that since $d$ is inner, we have
\be
\Delta d_\mathbf{r}=\Delta [\hbar\rho, -]=\hbar [\Delta (\rho), -]. 
\ee
It is not difficult to see that
\be
 \Delta (\rho)=\rho\otimes 1+1\otimes \rho-  \mathbf{r}-\mathbf{r}^{21}.
\ee
The result now follows from $[\mathbf{r}-\mathbf{r}^{21}, \Delta]=0$. 

\end{proof}

\begin{Cor}\label{Cor:coprodW}
    The element $W_\mathbf{r}$ satisfies
    \be
\Delta W_\mathbf{r}=W_\mathbf{r}\otimes 1+1\otimes W_\mathbf{r}+\Omega^2,
    \ee
    and consequently, $W= W_\mathbf{r}$ is independent of $(\fh_\pm)$. 
\end{Cor}

\begin{proof}
    Let us compute $\Delta[\rho, \rho]$. Using the fact that $\Delta(\rho)=\rho\otimes 1+1\otimes \rho-\mathbf{r}-\mathbf{r}^{21}$,  we have
    \be
\Delta [\rho, \rho]=[\Delta(\rho), \Delta (\rho)]=[\rho\otimes 1+1\otimes \rho-2\mathbf{r}, \Delta (\rho)].
    \ee
    Here the second identity follows from the fact that $[\Omega, \Delta]=0$. Similarly, we can write
    \be
[\rho\otimes 1+1\otimes \rho-2\mathbf{r}, \Delta (\rho)]=[\rho\otimes 1+1\otimes \rho-2\mathbf{r}, \rho\otimes 1+1\otimes \rho-2\mathbf{r}+\Omega].
    \ee
    The first identity of equation \eqref{eq:dr=r2} implies that
    \be
[\rho\otimes 1+1\otimes \rho-2\mathbf{r}, \rho\otimes 1+1\otimes \rho-2\mathbf{r}]=[\rho, \rho]\otimes 1+1\otimes [\rho, \rho], 
    \ee
    and therefore 
\be
\begin{aligned}
   \Delta[\rho, \rho]& = [\rho\otimes 1+1\otimes \rho-2\mathbf{r}, \rho\otimes 1+1\otimes \rho-2\mathbf{r}+\Omega]\\ &=[\rho, \rho]\otimes 1+1\otimes [\rho, \rho]+[\Delta(\rho)-\Omega,\Omega]\\ &=[\rho, \rho]\otimes 1+1\otimes [\rho, \rho]-2\Omega^2.  
\end{aligned}
\ee
This implies that $\Delta[\rho, \rho]=[\rho, \rho]\otimes 1+1\otimes [\rho, \rho]-2\Omega^2$. Since $[\rho, \rho]=2c_\mathbf{r}-2W_{\mathbf{r}}$ and $c\in\fg$ is primitive, we get
\be
\Delta W_{\mathbf{r}}=W_{\mathbf{r}}\otimes 1+1\otimes W_{\mathbf{r}}+\Omega^2. 
\ee
The fact that $W$ is independent of Lagrangians follow from this, since given two different such $W$, say $W_{\mathbf{r}_i}$ for $i=1, 2$, their difference satisfies
\be
\Delta \lp W_{\mathbf{r}_1}-W_{\mathbf{r}_2}\rp= \lp W_{\mathbf{r}_1}-W_{\mathbf{r}_2}\rp\otimes 1+1\otimes \lp W_{\mathbf{r}_1}-W_{\mathbf{r}_2}\rp,
\ee
and therefore $W_{\mathbf{r}_1}-W_{\mathbf{r}_2}\in \fg$. But since $W\in \mathrm{Sym}^3(\fg)$, their difference must be zero. 

\end{proof}

This corollary not only shows that $W$ is canonically associated to $\fg$ and its bilinear form $\kappa$, it also shows us how to construct a monoidal structure on the category of modules of the curved differential graded algebra $(U(\fg)\lbb\hbar\rbb, \hbar^2 W)$.

\begin{Prop}
    The tuple $\mathcal{D}:=(\Delta, \hbar\Omega)$ is a map of CDGAs
    \be
\mathcal{D}: (U(\fg)\lbb\hbar\rbb, \hbar^2 W)\to (U(\fg)\lbb\hbar\rbb, \hbar^2 W)\Chtensor (U(\fg)\lbb\hbar\rbb, \hbar^2 W), 
    \ee
satisfying $\CD\otimes 1(\CD)=1\otimes \CD (\CD)$ and $\epsilon\otimes 1(\CD)=1\otimes \epsilon (\CD)=(\mathrm{Id}, 0)$, where $\epsilon: U(\fg)\lbb\hbar\rbb\to\Ch$ is the co-unit map. Consequently, the category $(U(\fg)\lbb\hbar\rbb, \hbar^2 W)\Mod$ has the structure of a monoidal category. 

\end{Prop}

\begin{proof}
    The fact that $\CD$ defines a map of CDGAs follows from Corollary \ref{Cor:coprodW}. The co-associativity identity is clear from the co-associativity of $\Delta$, as well as the fact that
    \be
\Delta \otimes 1 (\Omega)=\Omega^{13}+\Omega^{23},\qquad 1\otimes \Delta (\Omega)=\Omega^{12}+\Omega^{13}. 
    \ee
    The identity with co-unit map is clear. 
    
\end{proof}

One can say that $(\CD, \epsilon)$ makes $(U(\fg)\lbb\hbar\rbb, \hbar^2 W)$ into a co-unital co-algebra object in the category of CDGAs. This is enough to define a monoidal structure. Explicitly, given two modules $(M, d_M)$ and $(N, d_N)$, their tensor product is given by
\be
M\Chtensor N, \qquad d=d_M\otimes 1+1\otimes d_N+\hbar \Omega\cdot -. 
\ee
Note that this is not symmetric unless $\hbar=0$ since $\sigma\Omega=-\Omega$. We believe that a braiding does not exist for this category, but we don't know how to prove this. In any case, we propose that $(U(\fg)\lbb\hbar\rbb, \hbar^2 W)$ together with this co-algebra structure is a canonical quantization of the 1-shifted metric Lie algebra structure of $\fg$. 

\subsubsection{An equivalent description using the Lagrangians}\label{subsubsec:QDL}

We now give an equivalent description of $(U(\fg)\lbb\hbar\rbb, \hbar^2 W)$ using the Lagrangians. Fixing the transverse Lagrangians $\{\fh_\pm\}$, we have shown in Lemma \ref{Lem:d2=c} that $U(\fg)\lbb\hbar\rbb$ carries a differential $d_\mathbf{r}=\hbar [\rho, -]$, such that $d^2_\mathbf{r}=\frac{\hbar^2}{2}[[\rho, \rho], -]=[\hbar^2 c_\mathbf{r}, -]$. Clearly $d_\mathbf{r}c_\mathbf{r}=0$ since $d_\mathbf{r}c_\mathbf{r}=\hbar [\rho, c_\mathbf{r}]=\hbar [\rho, c_\mathbf{r}-W]=\frac{\hbar}{2}[\rho, [\rho, \rho]]$. Therefore, the triple $(U(\fg)\lbb\hbar\rbb, d_\mathbf{r}, \hbar^2 c_\mathbf{r})$ defines a CDGA. The following is now clear from the definition. 

\begin{Prop}
    The tuple $(\mathrm{Id}, \hbar\rho)$ gives an isomorphism
    \be
(U(\fg)\lbb\hbar\rbb, \hbar^2 W)\simeq (U(\fg)\lbb\hbar\rbb, d_\mathbf{r}, \hbar^2 c_\mathbf{r}).
    \ee
\end{Prop}

\begin{Rem}
    The induced functor from $(\mathrm{Id}, \hbar\rho)$ on the category of modules acts as follows. Given a module $(M, d_M)$ of $(U(\fg)\lbb\hbar\rbb, \hbar^2 W)$, the image of it under the functor is given by $(M, d_M+\hbar \rho\cdot -)$. The curvature is the correct one since $(d_M+\hbar \rho\cdot -)^2=\hbar^2 W+\hbar^2 (c_\mathbf{r}-W)=\hbar^2 c_\mathbf{r}$. 
\end{Rem}

The following is now a quick consequence of equation \eqref{eq:dr=r2}, whose proof we simply omit. 

\begin{Prop}
The tuple $\CD_{\mathbf{r}}=(\Delta, -2\hbar \mathbf{r})$ is a map of CDGAs:
\be
\CD_{\mathbf{r}}: (U(\fg)\lbb\hbar\rbb, d_\mathbf{r}, \hbar^2 c_\mathbf{r})\to (U(\fg)\lbb\hbar\rbb, d_\mathbf{r}, \hbar^2 c_\mathbf{r})\Chtensor (U(\fg)\lbb\hbar\rbb, d_\mathbf{r}, \hbar^2 c_\mathbf{r}),
\ee
satisfying $\CD_{\mathbf{r}}\otimes 1(\CD_{\mathbf{r}})=1\otimes \CD_{\mathbf{r}} (\CD_{\mathbf{r}})$ and $\epsilon\otimes 1(\CD_\mathbf{r})=1\otimes \epsilon (\CD_{\mathbf{r}})=(\mathrm{Id}, 0)$. Moreover, the map $(\mathrm{Id}, \hbar\rho)$ is an isomorphism of co-algebra objects. 
  
\end{Prop}

The monoidal structure induced by $\CD_{\mathbf{r}}$ sends two modules $(M, d_M)$ and $(N, d_N)$ to
\be
M\Chtensor N, \qquad d=d_M\otimes 1+1\otimes d_N-2\hbar \mathbf{r}\cdot-. 
\ee
Again this structure is not symmetric monoidal because $\mathbf{r}$ is not symmetric. With the choice of the Lagrangians, the 1-shifted $r$-matrix enters into the definition of the coproduct, in comparison to the case of ordinary $r$-matrix, which enters into the definition of the braiding.

\subsection{Quantization of the Lagrangians}\label{subsec:quantization}

In this section, we consider the quantization of the Lagrangians $\fh_\pm$ . However, we will put the following assumption on the 1-shifted double $\fg$. 

\begin{Asp}\label{Asp:g2=0}
    We assume that $\fg_2=0$, or equivalently $\fg_{-1}=0$. 
\end{Asp}

\begin{Rem}
    This is obviously satisfied by $\fg=T^*[-1]\fh$ for some $\fh$ with $\fh_{-1}=0=\fh_2$. In fact for any $\fg$ satisfying $\fg_2=0$, it contains a subalgebra $T^*[-1]\fg_0$. 
\end{Rem}

In this case, we present the following theorem. 

\begin{Thm}\label{Thm:QDGAL}
    Let $(\fg, \fh_\pm)$ be a 1-shifted Manin-triple and assume \ref{Asp:g2=0}. The following statements are true. 
    
    \begin{itemize}
        \item The DG algebra $U_\hbar (\fg):= (U(\fg)\lbb\hbar\rbb, d_\mathbf{r})$ of Section \ref{subsubsec:QDL} contains two DG subalgebras $U_\hbar (\fh_\pm)$, generated by $\fh_\pm$ over $\Ch$. 

        \item The DG category $U_\hbar (\fh_+)\Mod$ (resp. $U_\hbar (\fh_-)\Mod$) is a right (resp. left) module category of $U_\hbar (\fg)\Mod$. 

        \item Let $U_\hbar (\hbar\fh_-)$ be the subalgebra of $U_\hbar (\fh_-)$ generated by $\hbar \fh_-$ over $\Ch$. The algebras $U_\hbar (\fh_+)$ and $U_\hbar (\hbar \fh_-)$ are Koszul dual to each other over $\Ch$. 
    \end{itemize}
\end{Thm}

\begin{Rem}
    We could have simply required that $c_{\mathbf r}=0$, but for all applications we care about, $\fg_2=0$, \textit{cf.} Section \ref{sec:meromorphic}. 
\end{Rem}

The precise meaning of the Koszul duality of the last point will be discussed in Section \ref{subsec:Koszul}. The proof of this Koszul duality (Proposition \ref{Prop:Koszul}) shows that $U_\hbar (\hbar \fh_-)$ is a deformation of the Chevalley-Eilenberg cochain complex of $\fh_+$, and the associated Poisson structure of this deformation is precisely the cobracket on $\fh_+$. This is another reason why we call the algebras $U_\hbar (\fh_\pm)$ quantizations. The rest of this section focuses on proving these statements.

\subsubsection{DG algebras and monoidal action}\label{subsubsec:monoidalact}

The construction of the DG algebras are now straighforward. With the assumption $\fg_2=0$, we must have $c_{\mathbf{r}}=0$. Consequently $U(\fg)\lbb\hbar\rbb$ admits a square-zero differential $d_\mathbf{r}$. Since $d_\mathbf{r}$ maps $\fh_\pm$ to $\mathrm{Sym}^2(\fh_\pm)$, it restricts to a square-zero differential on the subalgebras $U(\fh_\pm)\lbb\hbar\rbb$, supplying the DG algebras $U_\hbar (\fh_\pm)$. 

To define the monoidal action, we notice that since $\mathbf{r}\in \fh_+\otimes \fh_-$, the map $\CD_{\mathbf{r}}= (\Delta, -2\hbar \mathbf{r})$ has a well-defined restriction
\be
\CD_{\mathbf{r}}^+: U_\hbar (\fh_+)\to U_\hbar (\fh_+)\Chtensor U_\hbar (\fg), \qquad \CD_{\mathbf{r}}^-: U_\hbar (\fh_-)\to U_\hbar (\fg)\Chtensor U_\hbar (\fh_-). 
\ee
These satisfy $\CD_{\mathbf{r}}^\pm\otimes 1(\CD_{\mathbf{r}}^\pm)=1\otimes \CD_{\mathbf{r}}(\CD_{\mathbf{r}}^\pm)$ and $\epsilon\otimes 1(\CD_{\mathbf{r}}^\pm)=1\otimes \epsilon (\CD_{\mathbf{r}}^\pm)=(\mathrm{Id}, 0)$. Therefore they induce the monoidal actions of $U_\hbar (\fg)\Mod$ on $U_\hbar (\fh_\pm)\Mod$ respectively.

\subsubsection{Koszul duality from polarization}\label{subsec:Koszul}

We now show that the algebras $U_\hbar (\fh_+)$ and $ U_\hbar (\hbar\fh_-)$ are Koszul dual to each other. We will first discuss in what sense are these two algebras Koszul dual to each other. For this, we briefly recall the theory of Koszul duality for augmented algebras and coalgebas. For a detailed discussion of the constructions to appear below, see \cite{loday2012algebraic}. 

Let $A$ be a DG algebra over $\C$ with augmentation $\epsilon: A\to \C$. We denote by $B(A)$ the bar construction of $A$, which is given by
\be
B(A)=T^c (s \overline{A}),\qquad d=d_1+d_2.
\ee
Here $\overline{A}=\mathrm{Ker}(\epsilon)$, $s$ is the suspension functor, $T^c$ denotes the co-free tensor co-algebra generated by $s\overline{A}$, $d_1$ is generated by the differential on $A$ and $d_2$ is the Hochschild differential. 

Similarly, given $C$ a DG co-algebra with unit $1\to C$, one can construct the cobar construction $\Omega(C)$, which is given by
\be
\Omega(C)=T (s^{-1} \overline{C}),\qquad d=d_1+d_2.
\ee
Here $\overline{C}=C/1$, and $T$ denotes the free tensor algebra generated by $s^{-1} \overline{C}$. The elements $d_1$ and $d_2$ are similarly defined as above. 

Let $C$ be a DG co-algebra and $A$ a DG algebra, consider the vector space $\Hom (C, A)$. This has the structure of a DG algebra given by
\be
f* g:= \mu \circ (f\otimes g)\circ \Delta
\ee
where $\mu$ is the product for $A$ and $\Delta$ is the coproduct for $C$. If $C$ is a perfect complex in $\mathrm{Vect}$, then we can alternatively identify this with the product algebra $C^*\otimes A$. This DG algebra is called the convolution algebra of $C$ with $A$, and we denote it by $\{C, A\}$. 

A Maurer-Cartan element in a DG algebra $A$ is a degree $1$ element $\alpha\in A$ such that
\be
d_A \alpha+ \alpha^2=0.
\ee

\begin{Def}
    A twisting morphism in $\Hom (C, A)$ is a Maurer-Cartan element in the DG algebra $\{C, A\}$. The set of twisting morphisms is denoted by $\mathrm{Tw}(C, A)$. 
    
\end{Def}

The relation between $\mathrm{Tw}(C, A)$ and bar-cobar constructions is given by the following adjunction property. 

\begin{Prop}
    There are natural bijections
    \be
\Hom_{\rm Alg} \lp  \Omega (C), A\rp\cong \mathrm{Tw}(C, A)\cong \Hom_{\rm coAlg} (C, B (A)). 
    \ee
\end{Prop}

What this statement means is that $B(A)$ is the universal coalgebra classifying twist constructions. Indeed, there exists a universal twisting element
\be
\alpha_{univ}\in \mathrm{Tw}(B(A), A)
\ee
given by the identity algebra homomorphism on $B(A)$. For any other coalgebra $C$ with a twist $\alpha\in \mathrm{Tw}(C, A)$, there exists a unique morphism of DG algebras $f_\alpha: C\to B(A)$ such that $\alpha= \alpha_{univ}\circ f_\alpha$. The twisted tensor product $B(A)\otimes^\alpha A$, where the differential is deformed by $\alpha$, is a resolution of the trivial module $\C$ of $A$. The Koszul dual of $A$, if exists, must be given by an appropriate dual of $B(A)$. 

The above can all be generalized to algebras and coalgebras over $\C\lbb\hbar\rbb$, with which we now explain the meaning of the Koszul duality. The algebra $U_\hbar (\fh_+)$ is not finite-dimensional over $\Ch$ and so it does not make sense to talk about the linear dual of $B(U_\hbar (\fh_+))$. The algebra $U_\hbar (\hbar \fh_-)$ is also not finite-dimensional so it also does not make sense to take its dual. However, it does have a natural graded dual, which is given by $\mathrm{Sym}(\fh_+[1])\lbb\hbar\rbb$. Here we use the pairing between $\fh_-$ and $\fh_+[1]$ so that $(x_a, \hbar \epsilon^a)=1$. It is easy to see that the algebra structure of $U_\hbar (\hbar \fh_-)$ induces a coalgebra structure on $\mathrm{Sym}(\fh_+[1])\lbb\hbar\rbb$. This coalgebra structure does not need a completion of the symmetric powers of $\fh_+[1]$ thanks to the extra $\hbar$ in the pairing. We claim the following statement, which is the content of the Koszul duality statement of Thereom \ref{Thm:QDGAL}. 

\begin{Prop}\label{Prop:Koszul}
    There is a quasi-isomorphism of coalgebras
    \be
\mathrm{Sym}(\fh_+[1])\lbb\hbar\rbb\cong B(U_\hbar (\fh_+)).
    \ee
    
\end{Prop}

The proof of this will occupy the rest of this section. 

The trick here is the polarization of the differential $d_\mathbf{r}$. Let us consider the product algebra
\be
\wt U:= U_\hbar (\fh_+)\Chtensor U_\hbar (\hbar \fh_-)
\ee
with the differential $d_\mathbf{r}\otimes 1+1\otimes d_\mathbf{r}$. What we have shown is that the element $\alpha:= -2\hbar \mathbf{r}$ is a Maurer-Cartan element in this DG algebra. Using the duality between $ U_\hbar (\hbar \fh_-)$ and $\mathrm{Sym}(\fh_+[1])\lbb\hbar\rbb$, it is not difficult to see that this same $\alpha$ gives rise to a twisting homomorphism
\be
\alpha\in \mathrm{Tw} \lp \mathrm{Sym}(\fh_+[1])\lbb\hbar\rbb,  U_\hbar (\fh_+) \rp,
\ee
and consequently a coalgebra homomorphism
\be
f_\alpha: \mathrm{Sym}(\fh_+[1])\lbb\hbar\rbb\to B (  U_\hbar (\fh_+)).
\ee
We are left to show that this is a quasi-isomorphism. 

To do so, note that $f_\alpha$ induces a map of $ U_\hbar (\fh_+)$ modules
\be\label{eq:falphatwist}
\mathrm{Sym}(\fh[1])\lbb\hbar\rbb\Chtensor^\alpha  U_\hbar (\fh_+) \to B ( U_\hbar (\fh))\Chtensor^{\alpha_{univ}}  U_\hbar (\fh_+)
\ee
where $\Chtensor^\alpha$ means the tensor product twisted by the MC element $\alpha$. We just need to show that this induced map is a quasi-isomorphism. Indeed, if this is the case, then we find
\be
\mathrm{Sym}(\fh_+[1])\lbb\hbar\rbb\cong  \lp \mathrm{Sym}(\fh_+[1])\lbb\hbar\rbb\Chtensor^\alpha U_\hbar (\fh_+)\rp\otimes_{U_\hbar (\fh_+)}\Ch\cong B ( U_\hbar (\fh_+))=\mathrm{Tor}_{U_\hbar (\fh_+)}(\Ch). 
\ee
As we have mentioned, the complex $B ( U_\hbar (\fh_+))\Chtensor^{\alpha_{univ}} U_\hbar (\fh_+)$ is a resolution of $\Ch$, so we only need to show that the LHS of \eqref{eq:falphatwist} also resolves $\Ch$. We will use a spectral sequence argument. Let $C$ be this chain complex. Of course there are natural morphisms $\Ch\to C$ and $C\to \Ch$. The complex $C$ is equipped with a filtration
\be
C\supseteq \hbar C\supseteq \hbar^2 C\supseteq \cdots.
\ee
Thinking of this as a complex over $\C$, we obtain a spectral sequence, whose $E_1$ page is
\be
E_1^{p, q}=H^{p+q}(\mathrm{Gr}(C)).
\ee
After taking associated graded, the differential $1\otimes d$ drops off (since it has an $\hbar$ term), and $\alpha$ becomes $\sum \iota_{x_a}\otimes x_a$ (since the commutator of $U(\hbar \fh_-)$ has $\hbar$-term). A moment of consideration shows that $\mathrm{Gr}(C)$ is precisely the Chevalley-Eilenberg complex of $U(\fh_+)$ resolving $\C$. Therefore $E_1^{p, q}=\C$ for $p=-q$ otherwise it is zero. In particular, the complex terminates in this page and the full cohomology of $C$ is quasi-isomorphic to $E_1^{p,-p}$ as a vector space over $\C$. Now we can finish the proof since the clearly the associated $E_1$ page of the map $\Ch\to C$ is a quasi-isomorphism, implying that the map $\Ch\to C$ is a quasi-isomorphism. The proof of Proposition \ref{Prop:Koszul} is now complete.

\subsubsection{When the assumption is not satisfied}\label{subsec:failasp}

We give a quick discussion when $\fg_2\ne 0$. In this case, we have seen that there is an element $c_\mathbf{r}\in \fg_2$ such that $d^2=[c_\mathbf{r}, -]$. We can decompose $c_\mathbf{r}=c_\mathbf{r}^++c_\mathbf{r}^-$ for $c_\mathbf{r}^\pm\in \fh_\pm$. The fact that $c_\mathbf{r}$ preserves $\fh_\pm$ implies that
\be
[c^+_\mathbf{r}, \fh_-]\subseteq \fh_-, \qquad [c^-_\mathbf{r}, \fh_+]\subseteq \fh_+. 
\ee
Giving any $X, Y\in \fh_-$, we have
\be
0=\kappa ([c_\mathbf{r}, X], Y)=\pm \kappa (c_\mathbf{r}, [X, Y]). 
\ee
In particular, $\kappa (c_\mathbf{r}, -)$ defines Lie-algebra homomorphisms from $\fh_\pm\to \C$, which must be non-trivial. Let $\fh_\pm^c\subseteq \fh_\pm$ be the kernel of this map, which is an ideal. We claim now that $[c_{\mathbf{r}}^-, \fh_+]\subseteq \fh_+^c$ and $[c_{\mathbf{r}}^-, c_{\mathbf{r}}^-]=0$. The first statement is true because for any $X\in \fh_+$, we have
\be
\kappa (c_\mathbf{r}, [c_\mathbf{r}^-, X])=\kappa (c_\mathbf{r}, [c_\mathbf{r}, X])- \kappa (c_\mathbf{r}, [c_\mathbf{r}^+, X]).
\ee
The second term is zero since any commutator belong to $\fh_+^c$, and the first term is zero because
\be
\kappa (c_\mathbf{r}, [c_\mathbf{r}, X])=\pm \kappa ([c_\mathbf{r}, c_\mathbf{r}], X)=0,
\ee
which follows from $[c_\mathbf{r}, c_\mathbf{r}]=[[\rho, \rho], [\rho, \rho]]=0$. The second statement is true because for any $X\in \fh_+$, we have
\be
\kappa ([c_{\mathbf{r}}^-, c_{\mathbf{r}}^-], X)=\kappa (c_{\mathbf{r}}^-, [c_{\mathbf{r}}^-, X])=\kappa (c_\mathbf{r}, \fh_+^c)=0.  
\ee

\begin{Lem}
    The Lie subalgebras $\wt \fh_\pm=\fh_\pm^c\oplus \C c_{\mathbf{r}}^\mp$ are Lagrangian subalgebras of $\fg$. Moreover, $c_\mathbf{r}^\pm\in \fh_\pm^c$, namely $\kappa (c_\mathbf{r}^+, c_\mathbf{r}^-)=0$. 
\end{Lem}

\begin{proof}
    The fact that $\wt\fh_\pm$ are Lie subalgerbas follows from $[c_{\mathbf{r}}^\mp, \fh_\pm]\subseteq \fh_\pm^c$ and $[c_\mathbf{r}^\pm, c_\mathbf{r}^\pm]=0$. That they are Lagrangians follow from the fact that $\fh_\pm^c$ is defined to be the subspace having trivial pairing with $c_\mathbf{r}^\mp$. 

    Since $\wt \fh_\pm$ are both Lagrangians, and they have the same dimensions as $\fh_\pm$, they are at most transverse. However, if say $c_\mathbf{r}^+\ne 0$, then $\wt\fh_-$ loses one dimension in degree $-1$ from $\fh_-$, and grows 1-dimension in degree $2$, so it can't be that $\wt\fh_-\cap \wt\fh_+=0$. Therefore we must have $c_\mathbf{r}^\pm\in \fh_\pm^c$. 
\end{proof}

Let $\fh_\pm(c_\mathbf{r})$ be the Lie subalgebra of $\fg$ spanned by $\fh_\pm$ and $c_\mathbf{r}$, then these are two Lie subalgebras of $\fg$. The corresponding subalgebras $U(\fh_\pm (c_\mathbf{r}))\lbb\hbar\rbb$ are invariant under the differential $d_\mathbf{r}$. In fact, we have embeddings of CDGAs
\be
\iota_\pm: \lp U(\fh_\pm(c_\mathbf{r}))\lbb\hbar\rbb, d_\mathbf{r}, \hbar^2 c_\mathbf{r}\rp\longrightarrow \lp U(\fg)\lbb\hbar\rbb, d_\mathbf{r}, \hbar^2c_\mathbf{r}\rp.
\ee
The following proposition is clear. 

\begin{Prop}
    The CDGA maps $\CD_\mathbf{r}^\pm$ in Section \ref{subsubsec:monoidalact} defines monoidal actions
    \be
\lp U(\fh_+(c_\mathbf{r}))\lbb\hbar\rbb, d_\mathbf{r}, \hbar^2 c_\mathbf{r}\rp\Mod \Racts \lp U(\fg)\lbb\hbar\rbb, d_\mathbf{r}, \hbar^2c_\mathbf{r}\rp\Mod \Lacts \lp U(\fh_-(c_\mathbf{r}))\lbb\hbar\rbb, d_\mathbf{r}, \hbar^2 c_\mathbf{r}\rp\Mod
    \ee
    
\end{Prop}

In this case, although we can still define left and right module categories using the CDGAs $\lp U(\fh_\pm(c_\mathbf{r}))\lbb\hbar\rbb, d_\mathbf{r}, \hbar^2 c_\mathbf{r}\rp$, we are more hesitant to call these quantizations of $\fh_\pm$, since the algebras are larger. However, we expect these CDGAs to be related to quantizations of $\fh_\pm$.

\section{1-shifted Lie bialgebra structure from loop Lie algebras}\label{sec:meromorphic}

In this section, we will consider 1-shifted Lie bialgebra structures arising from splittings of loop Lie algebras. To do so, we must extend the results of Sections \ref{sec:shiftedlba} and \ref{sec:quantdouble} 
 to topological Lie algebras (since $\fg \lpp t\rpp$ is a topological Lie algebra). It turns out that this extension is slightly intricate, and care needs to be given regarding the topology. We will comment on this subtlety as we go. 

The main goal of this section is the construction of cohomologically-shifted analogue of the Yangian algebra. This \textit{DG 1-shifted Yangian} will be a DG algebra with the property that one can take tensor products of its modules in a ``meromorphic" fashion. We will specify what this means in Section \ref{subsec:1shiftedmeror}. The 1-shifted $r$-matrix will show up in this construction as a meromorphic 1-shifted $r$-matrix. The analogies between our construction and those of ordinary Yangians justifies the name ``DG 1-shifted Yangian". 

In what follows, all vector spaces involving $\fg\lpp t\rpp$ will come with a topology from loop grading.  We will denote by $\otimes$ the ordinary tensor product and $\wh\otimes$ the completed tensor product. When we consider left modules of $\fg\lpp t\rpp$, we mean left smooth modules of $\fg[t, t^{-1}]$, or left modules of the topological algebra $U(\fg\lpp t\rpp)$. Operations and maps will be continuous with respect to this topology. 

\subsection{Loop Lie algebras and its 1-shifted cotangent}

We start with some setups and notations. We will denote by $\CO=\C\lbb t\rbb$ the ring of Taylor series and $\CK=\C\lpp t\rpp$ the field of Laurent series. We denote by $\mathbb D=\mathrm{Spec}(\CO)$ the formal disk and $\mathbb D^\times= \mathrm{Spec}(\CK)$ the formal punctured disk.

We will let $\fg$ be a simple Lie algbera, and we denote by $\fd=T^*[-1]\fg$. Let $\fg (\CK)$ be the loop Lie algebra, which as a vector space is spanned by
\be
A(t)=\sum_n A_n t^n, \qquad A_n\in \fg, A_n=0 \text{ for } n\ll 0. 
\ee

Let us consider the 1-shifted cotangent $T^*[-1]\fg(\CK)$, where the dual is taken as the continuous dual with respect to the topology of loop grading. It can be identified with $\fd (\CK)$, and is a semi-direct product
\be
T^*[-1]\fg(\CK)=\fd (\CK)= \fg (\CK)\ltimes \fg^*(\CK)[-1].
\ee
The Lie algebra $\fg (\CK)$ has a natural Lie subalgebra $\fg (\CO)$. From this Lie subalgera, we obtain a 1-shifted Lagrangian of $\fd (\CK)$ of the form
\be
N^*[-1]\fg(\CO)\subseteq \fd (\CK).
\ee
Very explicitly, this Lie algebra is of the form
\be
N^*[-1]\fg(\CO)=\fg (\CO)\ltimes \fg^*(\CO)[-1]=\fd (\CO).
\ee
To define a 1-shifted Lie bialgebra structure, we need to find a transverse Lagrangian. We will do so with the help of a classical $r$-matrix. 

\subsubsection{Splittings of $\fg(\CK)$ from a classical $r$-matrix}

We will construct a Lagrangian transverse to $\fd(\CO)$ via the means of a Lie algebra splitting of $\fg(\CO)\subseteq \fg(\CK)$. This splitting will be provided by a classical $r$-matrix in the usual setting of Lie bialgebras. Let $\beta(-,-)$ be a non-degenerate bilinear form on $\fg$, with which we can define the quadratic Casimir $\Omega\in \fg\otimes \fg$ of $\fg$. The following formal expression
\be
\gamma(t_1, t_2)=\frac{\Omega}{t_1-t_2}.
\ee
is known as Yang's classical $r$-matrix and defines the structure of a Lie bialgebra on $\fg(\CO)$. In particular, it satisfies the so-called classical Yang-Baxter equation
\be\label{eq:CYBEgamma}
            [\gamma^{12}(t_1,t_2),\gamma^{13}(t_1,t_3)] + [\gamma^{12}(t_1,t_2),\gamma^{23}(t_2,t_3)] + [\gamma^{13}(t_1,t_3),\gamma^{23}(t_2,t_3)] = 0.
        \ee
One can think of $\gamma(t_1, t_2)$ as an element in $\fg(\CK)\otimes \fg(\CO)$ via the expansion
\be
\gamma (t_1, t_2)=\Omega\sum_{n\geq 0} \frac{t_2^n}{t_1^{n+1}}.
\ee
This classical $r$-matrix is related to the splitting of $\fg(\CK)$ into $\fg(\CO)\oplus t^{-1}\fg[t^{-1}]$, as the Lie algebra $t^{-1}\fg[t^{-1}]$ is spanned by elements coming from the first factor of $\gamma$. 

This relation between splittings and $r$-matrices are very general; see e.g.\ \cite{cherednik1983becklund, etingof1998lectures, skrypnyk2013infinite}. In particular, fix $\{b_i\}$ be a set of orthonomal basis for $\fg$, we can consider a series of the form
\be
r (t_1,t_2)= \frac{\Omega}{t_1-t_2}+ g(t_1, t_2),\qquad g(t_1, t_2) \in \fg(\CO)\otimes \fg(\CO).
\ee
Using the orthonormal basis, we always have a decomposition
\be
r (t_1,t_2)= \sum r_{k, i}(t_1)\otimes b_i t_2^k, \qquad r_{k, i}(t_1)=b_i t_1^{-k-1}+\mathrm{reg.}\in \fg(\CK).
\ee
The following statement is true, and for a proof, see \cite[Section 1]{abedin2022geometrization}.

\begin{Prop}\label{Prop:splitr}
 Let $\fg (r)$ be the subspace of $\fg (\CK)$ spanned by $r_{k, i}$.

 \begin{itemize}
     \item The subspace $\fg (r)$ is a complementary Lie subalgebra of $\fg (\CK)$ if $r$ satisfies the generalized Yang-Baxter equation
     \be\label{eq:CGYBE}
 [r^{12}(t_1,t_2),r^{13}(t_1,t_3)] + [r^{12}(t_1,t_2),r^{23}(t_2,t_3)] + [r^{32}(t_3,t_2),r^{13}(t_1,t_3)] = 0.
     \ee

     \item The equality \(\fg(r) = \fg(r)^\bot\), where the orthogonal complement is taken with respect to the bilinear form \((x,y) \mapsto \textnormal{res}_{t = 0}\beta(x(t),y(t))\), holds if and only if \(r\) is skew-symmetric, i.e.\ \(r(t_1,t_2) = -\sigma (r(t_2,t_1))\) where $\sigma$ is the flip map. In this case, \(r\) solves the 
        classical Yang-Baxter equation 
        \be\label{eq:CYBE}
            [r^{12}(t_1,t_2),r^{13}(t_1,t_3)] + [r^{12}(t_1,t_2),r^{23}(t_2,t_3)] + [r^{13}(t_1,t_3),r^{23}(t_2,t_3)] = 0
        \ee
        and \(r\) is simply called \(r\)-matrix with coefficients in \(\fg\). 

        \item The subalgebra \(\fg(r)\) is stable under the derivation \(\partial_t\) if and only if \(r\) depends on the difference \(t_1-t_2\) of its variables, i.e.\ \(r(t_1,t_2) = \Tilde{r}(t_1-t_2)\) for some \(\Tilde{r} \in (\fg \otimes \fg)(\!(t)\!)\). By abuse of notation, we simply write \(r(t_1,t_2) = r(t_1-t_2)\) in this case.
 \end{itemize}
\end{Prop}

Given such a splitting $\fg (\CK)=\fg (\CO)\oplus \fg (r)$, it is not difficult to derive the following result. 

\begin{Cor}

    Let $\fg$ be simple. For each solution $r(t_1, t_2)$ of the generalized classical Yang-Baxter equation, there exists a 1-shifted Manin-triple
    \be
N^*[-1]\fg (r)\oplus \fd (\CO)=\fd (\CK).
    \ee
    Here $N^*[-1]\fg (r)=\fg (r)\ltimes \fg (r)^\bot[-1]$. 
    
\end{Cor}

\begin{Rem}
    To avoid the clumbsy notation $N^*[-1]\fg (r)$, we will from now on denote it by $\fd(r)$, signifying its relation to $\fd=T^*[-1]\fg$. The vector space $\fd (r)$ has trivial topology, whereas $\fd (\CO)$ has non-trivial topology. 
    
\end{Rem}

Since the Lie algebras involved are infinite dimensional, it is not true that one can simply obtain cobrackets $\delta: \fh\to \fh\otimes \fh$ for both $\fh=\fd(\CO)$ and $\fh=\fd (r)$. We claim that the 1-shifted Manin-triple above induces cobrackets to the completed tensor products
\be
\delta_{\fd (r)}: \fd (r)\to \fd (r)\wh\otimes \fd (r)= \fd (r)\otimes \fd (r), \qquad \delta_{\fd (\CO)}: \fd (\CO)\to \fd (\CO)\wh\otimes \fd (\CO)=\fd\otimes \fd \lbb t_1, t_2\rbb.
\ee
The first cobracket follows from the fact that the bracket on $\fd (\CO)$ is graded by loop degree and can be obtained from $\fd [t]$. The graded dual of $\fd[t]$ is precisely $\fd (r)=\fd [t, t^{-1}]/\fd[t]$. This grading ensures $\delta_{\fd (r)}$ is valued in $\fd (r)\otimes \fd (r)$. The second cobracket follows from the fact that $\fd(r)$ is a filtered Lie algebra $\fd (r)=\varinjlim_i \fd (r)_{\leq i}$ whose filtered dual is precisely $\fd (\CO)$ with its topology. The proof of Proposition \ref{Prop:delta} applies to these brackets, and therefore they satisfy the relations
\be
\delta\otimes 1+1\otimes \delta (\delta)=0\in \mathrm{Sym}^3,\qquad \delta[X, Y]=[\delta X, Y]+(-1)^{|X|}[X, \delta Y]. 
\ee

\subsubsection{1-shifted $r$-matrix from classical $r$-matrix}

Associated to the Manin-triple $(\fd (\CK), \fd (r), \fd (\CO))$, one can write down an $r$-matrix 
\be
\mathbf{r}\in \fd(r)\wh\otimes \fd(\CO).
\ee
This is well-defined by the property that $1\otimes \kappa (\mathbf{r}, X)=-X$ for $X\in \fd (r)$. Of course it also satisfies $\kappa\otimes 1(\mathbf r, Y)=(-1)^{|Y|}Y$ for $Y\in \fd (\CO)$. Let $\delta=(\delta_{\fd (r)}, -\delta_{\fd(\CO)})$ be the co-bracket defined on $\fd (\CK)$:
\be
\delta: \fd (\CK)\to \fd (\CK)\wh\otimes \fd (\CK)=\fd\otimes \fd \lpp t_1, t_2\rpp,
\ee
then the proof of Proposition \ref{Prop:coboundaryr} applies here, and shows that
\be
\delta=[-\mathbf{r}, \Delta]\in \fd\otimes \fd \lpp t_1, t_2\rpp. 
\ee
Note that in a-priori the element $\mathbf{r}$ belongs to $\fd\otimes \fd \lpp t_1\rpp\lpp t_2\rpp$, the claim is that the commutator above in fact belongs to the subspace $\fd\otimes \fd \lpp t_1, t_2\rpp$.

We can relate $\mathbf{r}$ with the classical $r$-matrix $r(t_1, t_2)$ as follows. With the help of $\beta$, we can identify $\fg^*(\CK)$ with $\fg(\CK)$ as modules of $\fg(\CK)$. For any $X\in \fg(\CK)$, we will denote by $\epsilon X$ the corresponding element in $\fg^*(\CK)[-1]$. The elements $b_i t^k, \epsilon b_i t^k$ form a set of basis for $\fd (\CO)$. Since $r_{k, i}=b_i t^{-k-1}+\fg (\CO)$, the corresponding dual basis of $\epsilon b_i t^k$ under the pairing $\kappa$ is simply $r_{k, i}(t)$. The dual basis of $b_i t^k$ should be elements in $\epsilon \fg (r)^\bot$. To express these elements, we expand the element $r(t_1, t_2)$ over the region $|t_1|<|t_2|$ to find
\be
r (t_1, t_2)= \sum b_i t_1^k \otimes \wt r_{k, i}(t_2), \qquad \wt r_{k, i}\in \fg (\CK). 
\ee
The elements $\epsilon \wt r_{k, i}(t_2)$ are precisely the elements in $\epsilon \fg (r)^\bot$ dual to $b_i t^k$ (for a proof, see \cite[Lemma 1.24]{abedin2022geometrization}). We find in the end the following expression (where $\sigma$ is the flip map)
\be\label{eq:skewrr}
\mathbf{r}(t_1, t_2)=\sum  \epsilon \wt r_{k, i}(t_1)\otimes b_i t^k_2+ \sum r_{k,i}(t_1)\otimes \epsilon b_i t^k_2=1\otimes \epsilon (r(t_1, t_2))+\epsilon\otimes 1(\sigma r(t_2, t_1)).
\ee
We interpret this expression as an element in $\fd (r)\otimes \fd (\CO)$ by expanding the series at the region where $|t_1|>|t_2|$, so that $\frac{1}{t_1-t_2}=\sum \frac{t_2^n}{t_1^{n+1}}$. If we denote $\rho(t_1, t_2)=(1\otimes \epsilon) r (t_1, t_2)$, then the above can be rewritten as
\be\label{eq:rsymmetrize}
\mathbf{r}(t_1, t_2)= \rho (t_1, t_2)-\sigma (\rho (t_2, t_1)).
\ee
Here the minus sign is due to the fact that $\sigma (1\otimes \epsilon)=-(\epsilon\otimes 1)\sigma$. This expression bears analogy with the skew-symmetrization of the classical $r$-matrix in \cite[Proposition 2.1]{abedin2024yangiancotangentliealgebras}, except that now $\epsilon$ is of homological degree $1$. Of course this $\mathbf{r}$ satisfies classical Yang-Baxter equation of Proposition \ref{Prop:1shiftedYB}, where commutators are computed in $\fd^{\otimes 3}\lpp t_1\rpp\lpp t_2\rpp \lpp t_3\rpp$.  

\subsubsection{Consequence of the set-up}

The subtlety of topology becomes more severe when one tries to extend the result of Theorem \ref{Thm:QDGAL}. We can construct a differential $d_r$ on $U(\fd (\CK))\lbb\hbar\rbb$ in the same way, via the maps
\be
\btik
\fd (\CK)\rar{\hbar\delta} & \fd\otimes \fd \lpp t_1, t_2\rpp\lbb\hbar\rbb\rar{\nabla} & U(\fd (\CK))\lbb\hbar\rbb.  
\etik
\ee
Note that this multiplication here makes sense because the image of $\delta$ is in $\fd\otimes \fd \lpp t_1, t_2\rpp$. However, the multiplication from $\fd\otimes \fd\lpp t_1\rpp\lpp t_2\rpp$ to $U(\fd (\CK))\lbb\hbar\rbb$ does not necessarily make sense, so some care is needed when showing that the above differential is inner. Fortunately, the element $\mathbf{r}$ belong to $\fd\otimes \fd \lpp t_1\rpp\lbb t_2\rbb$, so for any smooth module, the action of $\nabla \mathbf{r}$ makes sense, and therefore $\nabla\mathbf{r}$ is well-defined in the universal enveloping algebra. It is now clear that $d=[-\nabla\mathbf{r}, -]$. However, the action of $\nabla\mathbf{r}^{21}$ does not make sense, since it has negative powers appearing before the positive powers, so one can't quickly use the fact that $\fd_2=0$ to deduce that $d^2=0$. However, suppose that $dX=X^{(1)}X^{(2)}$ for a possibly infinite (but convergent) sum, satisfying $X^{(1)}\otimes X^{(2)}=(-1)^{|X^{(1)}||X^{(2)}|}X^{(2)}\otimes X^{(1)}$, then we must have that $d^2$ applied to primitive elements belong to $\mathrm{Sym}^3\oplus \mathrm{Sym}^1$ (as in Remark \ref{Rem:d2odd}). By definition, the $\mathrm{Sym}^3$ part of this is still zero, and therefore $d^2X\in \fd (\CK)$. However, our algebra $\fd (\CK)$ is in degree $0, 1$, and so by degree reasons $d^2X=0$. 

In conclusion, we have DG algebras
\be
U_\hbar (\fd(\CO)),\qquad U_\hbar (\fd (r)),\qquad U_\hbar (\fd (\CK)),
\ee
such that the category $U_\hbar (\fd (\CK))\Mod$ is a monoidal category, and it acts on $U_\hbar (\fd(\CO))\Mod$ on the left, while acting on $U_\hbar (\fd(r))\Mod$ on the right. We note that by modules of $U_\hbar (\fd (\CK))$ we always mean smooth modules under the topology of loop grading, and similarly for the other algebras. The monoidal structure and the action are determined by the 1-shifted $r$-matrix $\mathbf{r}(t_1, t_2)$. 

\subsection{1-shifted meromorphic $r$-matrix}\label{subsec:1shiftedmeror}

The category $U_\hbar (\fd (\CK))\Mod$ is monoidal, while $U_\hbar (\fd(\CO))\Mod$ is not. This is simply because the element $\mathbf{r}$ is not valued in $\fd(\CO)$. In this section, we show that if we translate the coproduct, then one can perform a re-expansion to make $\mathbf{r}$ into an element in $\fd(\CO)$. This procedure of translating the coproduct to obtain $r$-matrices also show up in the study of ordinary Yangian's, as in \cite{gautam2021meromorphic}. From now on we will further assume that $r$ depends only on the difference of the arguments $r=r(t_1-t_2)$. Consequently, the element $\mathbf{r}=\mathbf{r}(t_1-t_2)$ is also difference-dependent. 

The algebras $U_\hbar (\fd(\CK))$ admits a differential $T$, which is given on generators by
\be
T x_{a, n}=nx_{a, n-1},\qquad T \epsilon^a_n=n\epsilon^a_{n-1}.
\ee
In short, $T x(t)=\pd_t x(t)$ and $T \epsilon (t)=\pd_t \epsilon (t)$.  

\begin{Lem}
  The differentials $T$ and $d_{\mathbf r}$ commutes with each other. 
\end{Lem}

\begin{proof}
    This follows from difference-dependence of $\mathbf{r}$. Indeed, we have
    \be
[T\otimes 1+1\otimes T, \mathbf{r}(t_1-t_2)]= \lp \pd_{t_1}+\pd_{t_2}\rp \mathbf{r}(t_1-t_2)=0.
    \ee
    Applying multiplication map $\nabla$ to this we get the desired result. 
    
\end{proof}

In particular, the differential $T$ is a differential on the DG algebra $U_\hbar (\fd (\CO))$ and $U_\hbar (\fd (r))$. Let us now consider the shifted coproduct (where $z$ is an auxiliary variable)
\be\label{eq:shiftedco}
U_\hbar (\fd (\CO))\longrightarrow U_\hbar (\fd (\CO))\wh \otimes_{\Ch} U_\hbar (\fd (\CO)) \lbb z\rbb
\ee
given by
\be
\Delta_z= \tau_z\otimes 1 (\Delta), \qquad \tau_z=e^{zT}. 
\ee
Similarly, one can shift the $r$-matrix into $(\tau_z\otimes 1)\mathbf{r}$ and re-expand over the region  where $|z|>|t_1|, |t_2|$ without changing the commutation relation. After this expansion, the element $(\tau_z\otimes 1)\mathbf{r}$ is of the form
\be
\mathbf{r}(t_1+z-t_2)=(1\otimes \epsilon) r (t_1+z -t_2)+ (\epsilon\otimes 1)\sigma (r(t_2-z-t_1)).
\ee
One can then hope that $(\Delta_z, -2\hbar (\tau_z\otimes 1)\mathbf{r})$ is a map of CDGAs. 

The above is not quite correct, since we run into the following issue. After re-expansion, the element $\mathbf{r}(t_1+z-t_2)$ lives in the space
\be
U_\hbar (\fd (\CO))\Chtensor U_\hbar (\fd (\CO))\lbb z^\pm\rbb,
\ee
and the expression $[\mathbf{r}(t_1+z-t_2), \Delta_z]$ does not make sense. To deal with this issue, we will \textbf{not} consider $\Delta_z$ as maps between algebras, but rather consider the corresponding structure on the category of modules. 

To this end, we introduce the type of modules we would like to consider. Let $M$ be a finitely generated smooth module (with respect to the topology of $\fd(\CO)$) of $U_\hbar (\fd (\CO))$ that is flat over $\C\lbb\hbar\rbb$. We assume that it is of the form $M=U\lbb\hbar\rbb$ for some vector space $U$, and that for some positive integer $K_M$ the elements $I_a t^k, \epsilon I_a t^k$ acts trivially on $M$ for $k\geq K_M$. We call these modules \textit{FSF} modules, which stands for finitely-generated, smooth and flat over $\C\lbb\hbar\rbb$. For any two such modules, say $M=U\lbb\hbar\rbb$ and $N=V\lbb\hbar\rbb$, their tensor product over $\C\lbb\hbar\rbb$ can be naturally identified with $U\otimes V\lbb\hbar\rbb$. 

\begin{Prop}\label{Prop:tensorz}
Let $M=U\lbb\hbar\rbb, N=V\lbb\hbar\rbb$ be two FSF modules of $U_\hbar (\fd (\CO))$. The following are true. 

    \begin{itemize}
        \item The differential:
        \be
d_{\mathbf{r}}(z):= d_M\otimes 1+1\otimes d_N-2\hbar \mathbf{r}(t_1+z-t_2)\cdot -
        \ee
        is a well-defined square-zero differential on the topological vector space
        \be
M\Chtensor N \lpp z\rpp.
        \ee
Here the topology is from loop space topology. We denote the resulting topological DG vector space by $M_z\otimes N_0$. 

\item The map $\Delta_z$ gives a well-defined action of $U_\hbar (\fd (\CO))$ on $M_z\otimes N_0$, although this module is no longer smooth. 

\item There is an isomorphism of DG modules
\be
M_z\otimes N_0\cong \tau_z\lp N_{-z}\otimes M_0\rp. 
\ee

    \end{itemize}
    
\end{Prop}

\begin{proof}
    Let us first prove 1. It is clear that $d_{\mathbf{r}}(z)$ is well-defined acting on $M\Chtensor N \lpp z\rpp$. Indeed, we expand the element $r(t_1+z-t_2)$ by
    \be\label{eq:zlarge}
r(t_1+z-t_2)= C\cdot \sum_{k\geq 0} \frac{(t_2-t_1)^k}{z^{k+1}}+ g(t_1+z-t_2)
    \ee
which induces the corresponding expansion of $\mathbf{r}$ using equation \eqref{eq:skewrr}. From the smoothness of $M, N$, we see that when acting on $M\Chtensor N$, $\mathbf{r}$ has finite order pole. 

For the rest of 1, we need to show that $d_{\mathbf{r}}(z)^2=0$, or in other words
\be\label{eq:dr=rrzlarge}
d_M\otimes 1+1\otimes d_N (\mathbf{r}(t_1+z-t_2))= \hbar [\mathbf{r}(t_1+z-t_2), \mathbf{r}(t_1+z-t_2)].
\ee
We know that the element $\mathbf{r}(t_1-t_2)$ satisfies
\be\label{eq:dr=rrgK}
d_{\mathbf{r}}\otimes 1+1\otimes d_{\mathbf{r}} (\mathbf{r}(t_1-t_2))= \hbar [\mathbf{r}(t_1-t_2), \mathbf{r}(t_1-t_2)].
\ee
This is an equation in $U_\hbar (\fd (r))\wh\otimes_{\Ch} U_\hbar (\fd (\CO))$. To make contact with equation \eqref{eq:dr=rrzlarge}, we consider the two formal expansions
\be
\iota_{t>z}: \C\lpp t-z\rpp\longrightarrow \C\lpp t\rpp\lbb z \rbb,\qquad \iota_{z>t}: \C\lpp t-z\rpp\longrightarrow \C\lpp z\rpp \lbb t\rbb.
\ee
These induce embeddings of Lie algebras
\be
\btik
 & \fd (\CK)\arrow[dr, "\iota_{z>t}\circ \tau_{-z}"] \arrow[dl, swap, "\iota_{t>z}\circ \tau_{-z}"] & \\
 \fd (\CK)\lbb z\rbb & & \fd \lpp z\rpp\lbb t\rbb 
\etik
\ee
Applying $\iota_{z>t}\circ \tau_{z}\otimes 1$ to equation \eqref{eq:dr=rrgK}, we obtain an equaiton
\be\label{eq:dr=rrzlarge'}
(\iota_{z>t}\otimes 1)\lp d_{\mathbf{r}}\otimes 1+1\otimes d_{\mathbf{r}} (\mathbf{r}(t_1+z-t_2))\rp= \hbar (\iota_{z>t}\otimes 1)\lp [\mathbf{r}(t_1+z-t_2), \mathbf{r}(t_1+z-t_2)] \rp
\ee
This is now an equation inside $ U (\fd \lpp z\rpp\lbb t\rbb )\lbb\hbar\rbb\wh\otimes_{\Ch} U(\fd (\CO))\lbb\hbar\rbb$, where we expand $\frac{1}{t_1+z-t_2}$ as
\be\label{eq:t2large}
\frac{1}{t_1+z-t_2}= -\sum_{k\geq 0} \frac{(t_2-t_1)^k}{z^{k+1}}.
\ee
Both sides of equation \eqref{eq:dr=rrzlarge'} can be applied to $M\otimes N$, but this is not the same as equation \eqref{eq:dr=rrzlarge}. The RHS is identical, the problem is that on the LHS of \eqref{eq:dr=rrzlarge} we have the action of $ (d_{\mathbf{r}} \iota_{z>t})\otimes 1 (\mathbf{r}(t_1+z-t_2))$ whereas on the LHS of \eqref{eq:dr=rrzlarge'} we have the action of $( \iota_{z>t} d_{\mathbf{r}}\otimes 1) (\mathbf{r}(t_1+z-t_2))$. Therefore, we need to show
\be
 (\iota_{z>t}d_{\mathbf{r}})\otimes 1(\mathbf{r})= (d_{\mathbf{r}} \iota_{z>t})\otimes 1 (\mathbf{r}). 
\ee
To do so, we consider now expanding the $\mathbf{r}(t_1+z-t_2)$ in the region where $t_1<t_2-z$, and think of this as an element in $\fd (\CO)\wh\otimes \fd (r)$. From the classical 1-shifted YB equation \eqref{eq:1shiftedYB}, this $\mathbf{r}$ satisfy
\be
(d_{\mathbf{r}} \iota_{t_1<t_2-z} )\otimes 1 \mathbf{r}(t_1+z-t_2)= \iota_{t_1<t_2-z}\otimes 1[\mathbf{r}(t_1+z-t_2)^{12}, \mathbf{r}(t_1+z-t_2)^{13}].
\ee
Re-expanding this using $\iota_{z>t_2}$ precisely land us at $(d_{\mathbf{r}} \iota_{z>t})\otimes 1 (\mathbf{r})$. Namely, we have
\be
 (d_{\mathbf{r}} \iota_{z>t})\otimes 1 (\mathbf{r})= \iota_{z>t_2, t_1} [\mathbf{r}(t_1+z-t_2)^{12}, \mathbf{r}(t_1+z-t_2)^{13}].
\ee
On the other hand, it is clear that
\be
(\iota_{z>t}d_{\mathbf{r}})\otimes 1 \mathbf{r}=\iota_{z>t_1, t_2}[\mathbf{r}(t_1+z-t_2)^{12}, \mathbf{r}(t_1+z-t_2)^{13}],
\ee
we find that they are equal.  

The second part is proven in a similar way, by translating the equality
\be
\Delta d_{\mathbf{r}}=(d_{\mathbf{r}}\otimes 1+1\otimes d_{\mathbf{r}}-2\hbar [\mathbf{r}, -])\Delta
\ee
by $\tau_z$, and re-expanding $\mathbf{r}$ into the region where $z>t_1, t_2$. This can be done because $\Delta (\nabla\mathbf{r})$ is well defined as an endomorphism of $M\otimes N$. The fact that $\iota$ are Lie algebra homomorphisms guarantees that this equality is still true after re-expansion. 

Finally, to prove the third part, we first note that
\be
\Delta_{-z}^{op}=1\otimes \tau_{-z}\Delta=\tau_z\otimes 1\Delta (\tau_{-z})=\Delta_{z}\tau_{-z}.
\ee
Therefore, the swapping map $\sigma: U\otimes V\to V\otimes U$ induces an isomorphism of modules of $U(\fd (\CO))$
\be
\sigma: M_z\otimes N_0\cong \tau_{z} \lp N_{-z}\otimes M_0\rp. 
\ee
We therefore only need to show that $\sigma$ intertwines the differential, or equivalently,
\be
\sigma (\mathbf{r}(t_1+z-t_2))= \mathbf{r}(t_1-z+t_2).
\ee
This follows from applying $\tau_z$ to equation \eqref{eq:rsymmetrize}, as well as the fact that $\mathbf{r}$ is difference-dependent. 

\end{proof}

\begin{Rem}
    The commutativity condition of Proposition \ref{Prop:tensorz} will be called weak commutativity. 
\end{Rem}

From Proposition \ref{Prop:tensorz}, we can define a ``meromorphic" tensor product $M_z\otimes N_0$ for a pair of FSF modules of $U_\hbar (\fd (\CO))$. However, the resulting module is not FSF anymore. Nevertheless, there is still a natural generalization of associativity condition as follows. Consider now a triple of FSF modules $M=U\lbb\hbar\rbb, N=V\lbb\hbar\rbb$ and $P=W\lbb\hbar\rbb$, the tensor producr $M_{z+w}\otimes (N_w\otimes P_0)$ is naturally identified with the vector space
\be
U\otimes V\otimes W\lbb\hbar\rbb \lpp w\rpp \lpp z+w\rpp
\ee
where as $(M_z\otimes N)_w\otimes P_0$ is identified with
\be
U\otimes V\otimes W\lbb\hbar\rbb \lpp z\rpp \lpp w\rpp.
\ee
These two are not immediately comparable, but they both receive a map from the vector space
\be
U\otimes V\otimes W\lbb\hbar\rbb \lbb z, w\rbb [z^{-1}, w^{-1}, (z+w)^{-1}]
\ee
via formal expansion, just as in \cite[Chapter 3]{frenkel2004vertex}. We claim the following result, which is mostly clear from the definition.

\begin{Prop}\label{Prop:weaklyassoc}
 For a triple of FSF modules $M=U\lbb\hbar\rbb, N=V\lbb\hbar\rbb$ and $P=W\lbb\hbar\rbb$, the following statements are true. 
 
 \begin{itemize}
     \item The vector space
\be
U\otimes V\otimes W\lbb\hbar\rbb \lbb z, w\rbb [z^{-1}, w^{-1}, (z+w)^{-1}]
\ee
has a well-defined DG module structure of $U_\hbar (\fd (\CO))$, whose action is induced by $(1\otimes \Delta_w) \Delta_{z+w}=(\Delta_z\otimes 1)\Delta_w$, and whose differential is given by
\be
d_{\mathbf{r}}(z, w):= d_M+d_N+d_P-2\hbar \lp \mathbf{r}^{12}(t_1+z-t_2)+\mathbf{r}^{13}(t_1+z+w-t_2)+\mathbf{r}^{23}(t_1+w-t_2)\rp.
\ee
We denote this module by $M_{z+w}\otimes N_w\otimes P_0$, signifying its independence of order. 

\item Formal expansion as discussed above gives the following maps of DG modules of $U_\hbar (\fd (\CO))$, both of which are equivalences after appropriate ring extension:
\be
\btik
 &M_{z+w}\otimes N_w\otimes P_0\arrow[dr]\arrow[dl] & \\
(M_z\otimes N)_w\otimes P_0 & & M_{z+w}\otimes (N_w\otimes P_0)
\etik
\ee

 \end{itemize} 

\end{Prop}

\begin{Rem}
   The associativity condition of Proposition \ref{Prop:weaklyassoc} will be called weak associativity. 
\end{Rem}

One can think of $M_z\otimes N_w\cong \tau_w(M_{z-w}\otimes N_0)$ as a sheaf of modules of $U_\hbar (\fd (\CO))$ over $\mathbb D^2$ away from the diagonal, where $z, w$ labels the insertion point. Weak commutativity guarantees that this sheaf only depends on the modules $M, N$ and their insertion points (not the order in which to form the tensor product). Moreover, the associativity condition is naturally an isomorphism between sheaves on $\mathbb D^3$ away from all the diagonals. 

In general, for a set of FSF modules $\{M^i\}$ of $U_\hbar (\fd (\CO))$ and formal variables $\{z_i\}$, one can define
\be
\bigotimes \{M^i, z_i\}:= M^1_{z_1}\otimes \cdots \otimes M^n_{z_n}=\prod V^i \lbb\hbar\rbb\lbb z_i\rbb [(z_i-z_j)^{-1}].
\ee
where $M^i=V^i\lbb\hbar\rbb$, whose differential is
\be
d_{\mathbf{r}}(z)=\sum_i d_{M_i}-2\hbar \sum_{i<j} \mathbf{r}^{ij}(t_i+z_i-z_j-t_j),
\ee
and whose action by $U_\hbar (\fd (\CO))$ is given by $\prod \tau_{z_i} \Delta^n$. This is a coherent sheaf over a product of disks $\mathbb D^n$ away from all the diagonals. It is in this sense that we can take tensor products of modules of $U_\hbar (\fd (\CO))$. 

\begin{Rem}
    The proof of Proposition \ref{Prop:tensorz} in fact shows that the re-expansion map:
    \be
\fd (\CK)\to \fd (\CO)\lbb z^{\pm}\rbb
    \ee
induces a functor
    \be
    U_\hbar (\fd (\CO))\Mod^{\rm FSF}\to U_\hbar (\fd (\CK))\Mod,
    \ee
    mapping $M$ to $M\lpp z\rpp$, whose differential is unchanged and whose action of $\fd (\CK)$ is given by the above re-expansion map. The tensor product $M_z\otimes N_0$ is nothing but the action of $M\lpp z\rpp$ on $N$, \textit{cf}. Theorem \ref{Thm:QDGAL}. There is a slight problem with this since $M\lpp z\rpp$ is not smooth. However, it is equal to the limit of a projective system of smooth modules. If one considers $M\lpp z\rpp$ in this sense, then one can think of the assignment $M\mapsto M\lpp z\rpp$ as the categorified state-operator correspondence. In this vein, one can view $U_\hbar (\fd (\CO))\Mod^{\rm FSF}$ as a categorified vacuum vertex algebra and the monoidal category $U_\hbar (\fd (\CK))\Mod$ as the categorification of the universal enveloping algebra of $U_\hbar (\fd (\CO))\Mod^{\rm FSF}$.

\end{Rem}

\subsection{Generalizing to affine Kac-Moody algebra}\label{subsec:affineKacr}

Let us now assume that the solution $r$ is skew-symmetric, so that $\fg (r)=\fg (r)^\bot$. In this case, we show that one can include a level in the algebra $U_\hbar (\fd (\CO))$. Consider a differential in $U_\hbar (\fd (\CO))$ given by
\be
d_k x_{a, n}= -\hbar n k (\epsilon x_{a, n-1}),
\ee
or, in other words, $d_k= \hbar k\epsilon T$. It is easy to show that this defines a differential. The proof of the following proposition will reveal the relation of $d_k$ to the affine Kac-Moody level.

\begin{Prop}
    $d_k^2=0$ and $(d_{\mathbf{r}}+d_k)^2=0$. We denote by $U_\hbar^k (\fd (\CO))$ the DG algebra $U (\fd (\CO))\lbb\hbar\rbb$ with differential $d_{\mathbf{r}}^k:=d_{\mathbf{r}}+d_k$. 
    
\end{Prop}

\begin{proof}
    One can check this by direct computation. However, we go through a different route, which reveals the connection between $d_k$ and the affine Kac-Moody level. Consider the affine Kac-Moody algebra $\wh \fg (\CK)$, which, as a vector space, is $\fg (\CK)\oplus \C K$. When $r$ is skew-symmetric, the Lie subalgebra $\fg (r)$ of $\fg (\CK)$ is, in fact, a Lie subalgebra of $\wh \fg (\CK)$ that is complementary to $\fg (\CO)\oplus \C K$. Let us consider the corresponding DG algebra obtained from $N^*[-1] \fg (\CO)\oplus \C K$ and the Manin-triple defined by $\fg (r)$. The underlying algebra is nothing but the algebra $U (\fd (\CO))\otimes \C[K]\lbb\hbar\rbb$. To understand the differential, we must understand the Lie algebra structure of $N^*[-1]\fg (r)$, which, as a vector space, is equal to
    \be
\fg (r)\oplus \epsilon \fg (r)\oplus \C \epsilon K^*. 
    \ee
    The subspace $\fg (r)\oplus \epsilon\fg (r)$ is a Lie subalgebra whose Lie bracket is precisely that of $N^*[-1]\fg (r)$ computed inside $T^*[-1]\fg (\CK)$. There is a non-trivial bracket between $\epsilon K^*$ and $\fg (r)$, due to the commutation relation
    \be
[x_{a, n}, x_{b, -n}]= n \kappa_{ab} K.
    \ee
This induces the following commutation relation in $N^*[-1]\fg (r)$:
\be
[r_{n, i}, \epsilon K^*] =(n+1) \epsilon r_{n+1, i}.
\ee
 A moment of reflection reveals that the cobracket induced on $U (\fd (\CO))\otimes \C[K]\lbb\hbar\rbb$ by this bracket is simply:
    \be
\delta_K b_it^n=  n K (\epsilon b_{i}t^{n-1}).
    \ee
Consequently, the full differential on $U (\fd (\CO))\otimes \C[K]\lbb\hbar\rbb$ is $d_{\mathbf{r}}+d_K$ where $d_K=-\hbar\delta_K$. We conclude the proof of the proposition by taking a quotient of this DG algebra by the ideal generated by $K-k$. 

\end{proof}

We now show that Proposition \ref{Prop:tensorz} and \ref{Prop:weaklyassoc} has a natural generalization to $U_\hbar^k (\fd (\CO))$.

\begin{Prop}
    For any two FSF modules $M, N$ of $U_\hbar^k (\fd (\CO))$, the following differential turns $M\Chtensor N\lpp z\rpp$ into a DG module of $U_\hbar^k (\fd (\CO))$ under the coproduct $\Delta_z$. 
    \be
d_{\mathbf{r}}(z):= d_M\otimes 1+1\otimes d_N-2\hbar \mathbf{r}(t_1+z-t_2)\cdot-.
    \ee
    Moreover, weak commutativity and associativity holds, in a fashion similar to  Proposition \ref{Prop:tensorz} and \ref{Prop:weaklyassoc}.

\end{Prop}

\begin{proof}
    The only nontrivial statements to check are the following equations
    \be
d_{\mathbf{r}}(z)^2=0,\qquad \Delta_z d_{\mathbf{r}}^k=d_{\mathbf{r}}(z) \Delta_z.
    \ee
    These boils down to the following equations
    \be
(\delta_k\otimes 1+1\otimes \delta_k) (\mathbf{r}(t_1+z-t_2))=0,\qquad (\delta_k\otimes 1+1\otimes \delta_k)\Delta_z=\Delta_z\delta_k.
    \ee
    The first equation above is because
    \be
(\delta_k\otimes 1+1\otimes \delta_k) (\mathbf{r}(t_1+z-t_2))= \hbar k (\pd_{t_1}+\pd_{t_2})(\epsilon\otimes \epsilon r (t_1+z-t_2))=0,
    \ee
    where we used the skew-symmetry and difference-dependence of $r$. The second equation is easily checked on generators. The proof of weak commutativity and associativity are evident from Proposition \ref{Prop:tensorz} and \ref{Prop:weaklyassoc}. 

\end{proof}

Similar to the case of $U_\hbar (\fd (\CO))$, we can now define
\be
\bigotimes \{M^i, z_i\}:= M^1_{z_1}\otimes \cdots \otimes M^n_{z_n}=\prod V^i \lbb\hbar\rbb\lbb z_i\rbb [(z_i-z_j)^{-1}].
\ee
for a set $\{M^i\}$ of FSF modules of $U_\hbar^k (\fd (\CO))$ and formal variables $z_i$. This is a coherent sheaf over $\mathbb D^n$ away from all the diagonals.

\subsection{Summary and some specializations}\label{subsec:summarymero}

Let us summarize what we have done so far. For a simple Lie algebra $\fg$, and each difference-dependent solution $r=r(t_1-t_2)$ of the classical generalized Yang-Baxter equation \eqref{eq:CGYBE}, we constructed a DG algebra $U_\hbar (\fd (\CO))$, and moreover show that it has the following property. For any set $\{M_i\}$ of FSF modules and a set of formal variables $\{z_i\}$, we can construct a sheaf over $\mathbb D^n$ away from all diagonals, of the form
\be
\bigotimes \{M^i, z_i\}:= M^1_{z_1}\otimes \cdots \otimes M^n_{z_n}=\prod V^i \lbb\hbar\rbb\lbb z_i\rbb [(z_i-z_j)^{-1}],
\ee
where $M^i=V^i \lbb\hbar\rbb$. This admits a natural differential
\be
d_{\mathbf{r}}(z)=\sum_i d_{M_i}-2\hbar \sum_{i<j} \mathbf{r}^{ij}(t_i+z_i-z_j-t_j),
\ee
and the structure of a $U_\hbar (\fd (\CO))$ module, given by the $\tau_{z_i}$-shift of the symmetric co-product. This tensor product is weakly commutative and associative. 

When $r=r(t_1-t_2)$ is moreover skew-symmetric, and satisfies the ordinary Yang-Baxter equation \eqref{eq:CYBE}, we show that the algebra $U_\hbar (\fd (\CO))$ can be further deformed by a differential to $U_\hbar^k (\fd (\CO))$, such that the above statement still holds for $U_\hbar^k (\fd (\CO))$. This $k$ is closely related to the affine Kac-Moody level. 

We now comment on two specializations. First, if we consider $r(t_1-t_2)$ that are moreover rational, in the sense that
\be
r(t_1-t_2)= \frac{\Omega}{t_1-t_2}+ g(t_1-t_2), \qquad g(t_1-t_2)\in \fg [t_1]\otimes \fg[t_2],
\ee
then it is not difficult to show that the above sheaf $\bigotimes \{M^i, z_i\}$ can be extended to a sheaf over $\C^n$ away from all diagonals. This is simply because in this case $d_{\mathbf{r}}(z)$ are all rational functions with poles at $z_i-z_j$.  One can in particular take the stalk at any $\{s_i\}$ where $s_i\ne s_j$ for $i\ne j$. Unfortunately this stalk itself is not a module of $U_\hbar (\fd (\CO))$. It is, however, a module of $U_\hbar (\fd [t])$  (which is indeed a DG subalgebra of $U_\hbar (\fd (\CO))$, thanks to rationality of $g(t_1-t_2)$). Therefore, in the case of rational $r$ matrices, starting with finite-dimensional modules of $U_\hbar (\fd [t])$, we can evaluate their tensor product over the configuration space of distinct points over $\C$. The same is true for $U_\hbar^k (\fd [t])$ when $r$ is moreover skew-symmetric. Namely, the category of finite-dimensional modules of $U_\hbar^k (\fd [t])$ has the structure of a meromorphic tensor category of \cite{soibelmanmero}.  In fact, we expect that in all cases considered in this section, the sheaf $\bigotimes \{M^i, z_i\}$ can be extended to the curve corresponding to the solution $r$ of the classical Yang-Baxter equation.   

We consider the second specialization. In this section, we have restricted our considerations to a simple Lie algebra $\fg$, because we would like to use well-known results about classical $r$ matrices. For any Lie algebra $\fg$, we can always form the Yang's $r$-matrix as an element in $\fd:= T^*[-1]\fg$, of the form
\be
\mathbf{r}(t_1-t_2)=\frac{I_a\otimes  I^a}{t_1-t_2}=\frac{\Omega}{t_1-t_2}\in t_1^{-1}\fd [t_1^{-1}]\otimes \fd \lbb t_2\rbb,\qquad  \fd=T^*[-1]\fg.
\ee
In this $I_a$ is a basis for $\fd$ and $I^a$ the dual basis under the 1-shifted pairing, and $\Omega$ is the quadratic Casimir of the 1-shifted pairing. This $\mathbf{r}$ is always rational, difference-dependent and skew-symmetric (in the sense that this splitting works for the affine Kac-Moody algebra as well, if there is an ordinary symmetric pairing on $\fd$ to define it). Due to the resemblance of this $\mathbf{r}$ with the ordinary Yang's $r$-matrix, we feel that it is right and just to call the associated DG algebra the \textbf{DG 1-shifted Yangian} of $\fd$, and denote it by ${}_1\!Y_\hbar^k (\fd)$. We can construct coherent sheaves over the configuration space of points on $\C$ from a collection of FSF modules of ${}_1\!Y_\hbar^k (\fd)$, much like the case of ordinary Yangians. 

\newpage

\bibliographystyle{alpha}

\bibliography{SQ}

\end{document}